\documentclass[a4paper, 11pt]{article}

\usepackage{amsthm,amssymb}
\usepackage{mathtools}		%IN TEXLIVE		% to redefine \abs
\usepackage{enumitem}		 	%IN TEXLIVE		% to make list labels customisable
\usepackage{graphicx} 		%IN TEXLIVE		%to include pdf images
\usepackage{subfigure} 		%IN TEXLIVE		%multiple figures under same number

\theoremstyle{plain}	
\newtheorem{thm}{Theorem}[section]
\newtheorem{lem}[thm]{Lemma}
\newtheorem{cor}[thm]{Corollary}
\newtheorem{claim}[thm]{Claim}
\newtheorem*{thm:RecursiveConstructionDBalMConn}{Theorem \ref{thm:RecursiveConstructionDBalMConn}}
\newtheorem*{thm:MConnectedGloballyRigidIffDBal}{Theorem \ref{thm:MConnectedGloballyRigidIffDBal}}

\theoremstyle{definition}
\newtheorem{case}{Case}[thm] %Numbers case as section.subsection.theorem.case
\DeclarePairedDelimiter\abs{\lvert}{\rvert}

\newcommand{\M}{\mathcal{M}}
\newcommand{\R}{\mathcal{R}}

\title{Global rigidity of 2-dimensional direction-length frameworks with connected rigidity matroids} 
\author{Katie Clinch \thanks{School of Mathematical Sciences, Queen Mary University of London, Mile End Road, London, E1 4NS, United Kingdom. Email: k.clinch@qmul.ac.uk}} 
\begin{document}

%\begin{frontmatter}%ARXIV -remove

\maketitle %ARXIV - added

\begin{abstract}
A two-dimensional direction-length framework $(G,p)$ consists of a multigraph $G=(V;D,L)$ whose edge set is formed of ``direction'' edges $D$ and ``length'' edges $L$, and a realisation $p$ of this graph in the plane. The edges of the framework represent geometric constraints: length edges fix the distance between their endvertices, whereas direction edges specify the gradient of the line through both endvertices. A direction-length framework $(G,p)$ is globally rigid if every framework $(G,q)$ which satisfies the same direction and length constraints as $(G,p)$ can be obtained by translating $(G,p)$ in the plane, and/or rotating $(G,p)$ by $180^{\circ}$. 

In this paper, we characterise global rigidity for generic direction-length frameworks whose associated rigidity matroid is connected, by showing that such frameworks are globally rigid if and only if every 2-separation of the underlying graph is direction-balanced. This extends previous work by Jackson and Jord{\'a}n \cite{JJ_MixedCircuits}, who considered direction-length frameworks where the edge set forms a circuit in the rigidity matroid.
\end{abstract}

%KEYWORDS
% direction-length frameworks, rigidity matroid, global rigidity
%----------------------- ARXIV -remove - S -----------------------
%\begin{keyword} 
%global rigidity, rigidity matroid, direction-length frameworks
%\end{keyword}
%----------------------- ARXIV -remove - E -----------------------

%\end{frontmatter} %ARXIV -remove

%#######################################################################################################
%#######################################################################################################

\section{Introduction}

%DEFN - direction-length framework
A two-dimensional \emph{direction-length framework} is a pair $(G,p)$, where $G=(V;D,L)$ is a loop-free multigraph whose edge set consists of ``direction'' edges $D$, and ``length'' edges $L$, such that any pair of vertices has at most one edge of each type between them, and $p: V \to \mathbb{R}^2$ is a realisation of the graph in the plane. We consider the edges of our framework $(G,p)$ to define geometric constraints: a direction edge $uv\in D$ fixes the gradient of the line through $p(u)$ and $p(v)$, whereas a length edge $uv\in L$ specifies the distance between the points $p(u)$ and $p(v)$. 

In the special cases where $D=\emptyset$ or $L=\emptyset$ we say that the graph $G$ and framework $(G,p)$ are \emph{length-pure} or \emph{direction-pure} respectively, and refer to these collectively as \emph{pure graphs} and \emph{pure frameworks}. When both $D$ and $L$ are non-empty we call $G$ a \emph{mixed graph} and $(G,p)$ a \emph{mixed framework}.

%DEFN: equivalence, congruence, global rigidity
Given such a configuration, $(G,p)$, a natural question to ask is whether there are other realisations of $G$ which satisfy the same direction and length constraints. Any framework $(G,q)$ which satisfies the same constraints as $(G,p)$ is said to be \emph{equivalent} to $(G,p)$. It is clear that we can always translate a framework in the plane and/or rotate it by $180^{\circ}$ to obtain an equivalent framework; any realisation that can be obtained in this manner from our original framework $(G,p)$, is said to be \emph{congruent} to $(G,p)$. If all frameworks which are equivalent to $(G,p)$ are also congruent to $(G,p)$ then we say that $(G,p)$ is \emph{globally rigid}. However many frameworks are not globally rigid, such as those in Figure \ref{fig:NecessaryConditionsButNotGloballyRigid}.

%*********************************************************
%  Satisfies necessary conditions but not globally rigid
%*********************************************************
\begin{figure}
\centering
\includegraphics[scale = 1]{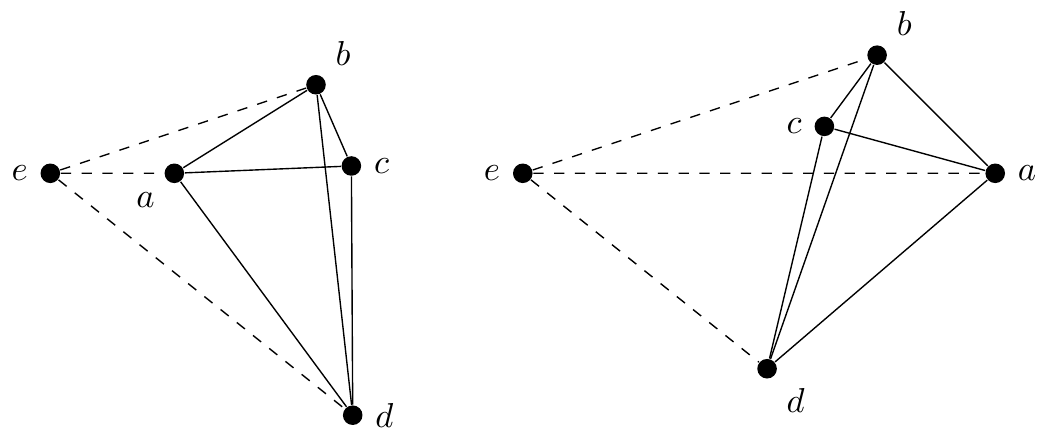}
\caption{Two equivalent but non-congruent realisations of a mixed graph in $\mathbb{R}^2$. Direction constraints are represented by dashed lines, and length constraints by solid lines.}
\label{fig:NecessaryConditionsButNotGloballyRigid}
\end{figure}

%MOTIVATION - applications
Global rigidity is an important property in many real-world applications. For example, in Computer Aided Design (CAD), a planar design consists of a collection of geometric objects such as line segments, points and curves, with constraints on their size, and relative positions. It is still an open problem to determine when such a design can be uniquely identified by its constraints. Similarly, in the theory of sensor networks, we would like to know when it is possible to determine the position of individual sensors, given their relative positions and the location of two transmitters. In both cases, we are asking whether the constraints are sufficient to ensure a unique realisation. In other words, whether the framework is globally rigid.

\subsection{Rigidity}

%TODO - add characterisation of rigidity of DL-graphs? Then can remove mixed condition of DL-frmaeworks as globally rigid => rigid => mixed.

%DEFN - rigidity, infinitesimal rigidity
A closely related concept, which is a necessary condition for global rigidity, is rigidity. A direction-length framework $(G,p)$ is \emph{rigid} if the only continuous motions of the vertices which preserve the constraints are translations of the entire framework. We say a mixed framework $(G,p)$ is \emph{redundantly rigid}, if $(G-e,p)$ is rigid for all edges $e\in E(G)$. 

It is often helpful to consider a stronger version of rigidity, known as infinitesimal rigidity, which is based on the \emph{rigidity matrix} $R(G,p)$ of the framework $(G,p)$. This matrix consists of $2\abs{V}$ columns and $\abs{D}+\abs{L}$ rows, where each row corresponds to an edge, and each successive pair of columns to a vertex of $G$. For any edge $uv\in E(G)$, the corresponding row of the rigidity matrix has zero-entries throughout, except in the two pairs of columns corresponding to the endvertices $u$ and $v$. These columns instead contain the vectors $p(u)-p(v)$ and $p(v)-p(u)$ respectively when $uv\in L$, and  $(p(u)-p(v))^\perp$ and $(p(v)-p(u))^\perp$ respectively when $uv\in D$ (where $(x,y)^\perp=(y,-x)$). We define a direction-length framework $(G,p)$  to be \emph{infinitesimally rigid} if $\text{rank}(R(G,p))=2\abs{V}-2$. 

Every motion of a framework $(G,p)$ starts as an ``instantaneous motion'' of that framework, i.e.\ an assignment of instantaneous velocity vectors to each of the vertices in the framework. These velocity vectors can be concatenated to give a single vector of length $2\abs{V}$ in the kernel of $R(G,p)$. All frameworks $(G,p)$, whether rigid or not, have at least two motions: the two translations. Hence $\abs{\text{ker}(R(G,p))}\geq 2$. If the framework is infinitesimally rigid, then this holds with equality, and so the only motions of the framework are translations. Thus any infinitesimally rigid framework is rigid. However, the converse is not true in general, since there may exist vectors in the kernel of $R(G,p)$ whose corresponding instantaneous motion of the framework does not extend to a finite motion.

%DEFN - Generic
One of the key questions in rigidity theory is to determine to what extent the rigidity properties of a framework are combinatorial (determined by the underlying graph), or geometric (dependent on the realisation of the framework). In general, this is not an easy question to answer. However for infinitesimal rigidity, it is much simpler: for all possible realisations of a graph $G$, the rank of the rigidity matrix will be maximised when there are no algebraic dependencies between the coordinates of the vertices. As such, we say that a realisation $p$ of a graph $G$ is \emph{generic} if the coordinates in $p(V)$ are algebraically independent over the rationals, and we call the corresponding framework $(G,p)$, a \emph{generic framework}. 

Given a direction-length framework $(G,p)$, the rows in the rigidity matrix $R(G,p)$ define a matroid. Further, any two generic realisations of $G$ will define the same matroid, which we call the \emph{rigidity matroid of G}, and denote $\R(G)$. Following standard matroid terminology, we say that a set of edges of a graph $G$ is \emph{independent} if the corresponding rows in $\R(G)$ are linearly independent, and that a set of edges is a \emph{circuit} if the corresponding rows are linearly dependent in $\R(G)$, but any proper subset is linearly independent.

Since either all generic realisations of a graph $G$ are infinitesimally rigid, or none of them are, we say that infinitesimal rigidity is a \emph{generic property} of direction-length frameworks. We can now define a graph $G$ to be \emph{infinitesimally rigid} when it has a generic realisation which is infinitesimally rigid (or equivalently, when all generic realisations are infinitesimally rigid). A similar argument to Asimow and Roth's for bar-and-joint frameworks \cite{AR_RigidityOfGraphs}, can be used to show that rigidity and infinitesimal rigidity are equivalent for generic direction-length frameworks (see \cite[Lemma 8.1]{JK_DL_BoundedFrameworks}). This implies that rigidity and redundant rigidity are also generic properties for direction-length frameworks. In contrast, it remains an open problem to determine whether global rigidity is a generic property for direction-length frameworks.

\subsection{Global Rigidity}

Generic global rigidity is not yet fully understood for direction-length frameworks. In comparison, generic global rigidity has been completely characterised for  direction-pure and length-pure frameworks. 

Pure frameworks have additional isometries which do not violate their constraints, and so we have to adapt our definitions for rigidity and global rigidity in pure frameworks accordingly. A length-pure framework has no constraints on its orientation, and so can be rotated or reflected, whereas a direction-pure framework lacks any distance constraints, and so can be dilated. 

For length-pure frameworks, more commonly called ``bar-and-joint frameworks'', we say that a framework $(G,p)$ is \emph{globally length-rigid} if all equivalent realisations $(G,q)$ can be obtained from $(G,p)$ by a translation, rotation or reflection. Gortler et al.\ \cite{G_L_CharacterisingGlobalRigidity} showed that global length-rigidity is a generic property for length-pure frameworks in $\mathbb{R}^d$ for all $d>0$.

We can similarly define a direction-pure framework $(G,p)$ to be \emph{globally direction-rigid} if all equivalent realisations can be obtained from $(G,p)$ by a translation or dilation. Whiteley \cite{W_MatroidsFromDiscAppGeom} showed that global direction-rigidity is equivalent to direction-rigidity for direction-pure frameworks in $\mathbb{R}^d$. This implies that global direction-rigidity is a generic property for direction-pure frameworks in $\mathbb{R}^d$ for all $d>0$.

These results suggest that global rigidity might also be a generic property for direction-length frameworks in $\mathbb{R}^d$, but this is not yet known. So for our study of direction-length frameworks in the plane, we shall define a direction-length graph $G$ to be \emph{globally rigid} if all generic realisations of $G$ in $\mathbb{R}^2$ are globally rigid. Several combinatorial conditions for global rigidity in direction-length frameworks have already been identified:

\begin{lem}\cite[Lemma 1.6]{JJ_MixedCircuits},\cite[Theorems 1.1 and 1.3]{JK_NecessaryConditions} \label{lem_JJK:NecessaryConditionsForGlobalRigidity} %Lemma 1.6
Let $(G,p)$ be a generic direction-length framework with at least three vertices. Suppose that $(G,p)$ is globally rigid in $\mathbb{R}^2$, and let $G=(V;D,L)$ Then
\begin{enumerate}
	\item $G$ is mixed,
	\item $G$ is rigid,
	\item $G$ is 2-connected,
	\item $G$ is direction-balanced, i.e. both sides of any 2-separation contain a direction edge,
	\item the only 2-edge-cuts which can occur in $G$ consist of two direction edges incident with a common vertex of degree two,
  \item if $|L|\geq 2$, then $G-e$ is rigid for all $e\in L$, and
	\item if $e\in D$ and $G-e$ contains a rigid subgraph on at least 2 vertices, then $G-e$ is either rigid or unbounded.
\end{enumerate}
\end{lem}

Where a direction-length framework $(G,p)$ is \emph{unbounded} if there exists an equivalent framework $(G,q)$ such that for some $u,v\in V(G)$, $\| q(u)-q(v)\|> K$ for all $K\in \mathbb{R}$. The conditions in Lemma \ref{lem_JJK:NecessaryConditionsForGlobalRigidity} are not sufficient to guarantee global rigidity. There exist generic direction-length frameworks, such as those in Figure \ref{fig:NecessaryConditionsButNotGloballyRigid}, which satisfy all of the above conditions, but are not globally rigid. 

A problem which is closely related to identifying whether global rigidity is a generic property, is to characterise all globally rigid frameworks in terms of the underlying graph. Jackson and Jord\'{a}n succeeded in characterising global length-rigidity for generic length-pure frameworks:

\begin{thm}\cite[Theorem 7.1]{JJ_LengthConnected} %Theorem 7.1
A generic length-pure framework $(G,p)$ where $G=(V,L)$ is globally length-rigid in $\mathbb{R}^2$ if and only if either $G$ is a complete graph on at most 3 vertices, or $G$ is 3-connected and redundantly length-rigid.
\end{thm}

Later, they applied a similar method to direction-length frameworks, and characterised global rigidity for a class of generic direction-length frameworks: those whose edge set is a circuit in the rigidity matroid.

\begin{thm}\cite[Theorem 6.2]{JJ_MixedCircuits} \label{thm_JJ:GloballyRigidCircuits}%Theorem 6.2
Let $(G,p)$ be a generic realisation of a mixed graph whose rigidity matroid is a circuit. Then $(G,p)$ is globally rigid if and only if $G$ is direction-balanced.
\end{thm}

%MOTIVATION - Intent
A circuit is the simplest instance of a ``connected matroid'', which we shall define formally in Section \ref{sec:M-Connected}. In this paper, we extend Jackson and Jord\'{a}n's result to all generic direction-length frameworks with a connected rigidity matroid. To do this, we first find an inductive construction of all such graphs which are direction-balanced: 

\begin{thm:RecursiveConstructionDBalMConn}
Let $G$ be a mixed graph. Then $G$ is a direction-balanced mixed graph with a connected rigidity matroid if and only if $G$ can be obtained from $K_3^+$ or $K_3^-$ by a sequence of edge additions, 1-extensions and 2-sums with direction-pure $K_4$'s.
\end{thm:RecursiveConstructionDBalMConn} 

Obtaining this construction is the main theme of the paper. The inductive moves are described in Sections \ref{subsec:Definitions_GlobalRigidityOperations} and \ref{subsec:2Sums}. Our extension of Theorem \ref{thm_JJ:GloballyRigidCircuits} follows readily from this inductive construction, by using properties of the moves used.

\begin{thm:MConnectedGloballyRigidIffDBal}
Let $p$ be a generic realisation of a mixed graph $G$ with a connected rigidity matroid. Then $(G,p)$ is globally rigid if and only if $G$ is direction-balanced. 
\end{thm:MConnectedGloballyRigidIffDBal}

Theorem \ref{thm:MConnectedGloballyRigidIffDBal} implies that the inductive construction in Theorem \ref{thm:RecursiveConstructionDBalMConn} is also a construction of all globally rigid $\M$-connected graphs. See Theorem \ref{thm:RecursiveConstructionGloballyRigidMConn}. As a corollary, we deduce that global rigidity is a generic property for frameworks with connected rigidity matroids. 

%SUMMARY
The paper is structured as follows. In Section \ref{sec:Definitions} we review some elementary results from rigidity theory, before introducing $\M$-connected graphs (graphs whose rigidity matroid is connected) in Section \ref{sec:M-Connected}. Section \ref{sec:Circuits} then builds upon Jackson and Jord{\'a}n's previous work on circuits. This enables us to find a recursive construction of all $\M$-connected mixed graphs in Section \ref{sec:Admissible}, by generalising ideas from \cite{JJ_LengthConnected}; leading to our main result in Section \ref{sec:Feasible}: an inductive construction of all direction-balanced, $\M$-connected mixed graphs. Finally, in Section  \ref{sec:GlobalRigidity}, we show that this result characterises global rigidity for $\M$-connected graphs, before briefly describing how this result fits into the work to characterise global rigidity for all direction-length graphs in Section \ref{sec:Further}.

%#######################################################################################################
%#######################################################################################################

\section{Preliminaries}\label{sec:Definitions}

%INTRODUCTION
We first introduce some notation before reviewing key results from rigidity theory. 
%DEFN: E(X)
Let $G=(V;D,L)$. When $\emptyset \neq X\subseteq V$, we let $E_G(X)$ denote the edge set of $G[X]$, and  abbreviate this notation to $E(X)$ when it is clear which graph we are referring to.
%DEFN: independent, circuit
It will be helpful to extend this notation to also describe subgraphs induced by sets of edges. When $C\subseteq D\cup L$ is a non-empty set of edges, we let $G[C]$ denote the subgraph induced by $C$, which has edge set $C$ and vertex set $\{v\in V : uv\in C \text{ for some } u\in V\}$.

%DEFN: i(X), i_D(X), i_L(X)
In what follows, we shall frequently wish to count the numbers of edges in mixed or pure subgraphs. To that end, given a graph $G=(V;D,L)$ and vertex set $X\subseteq V$, we let $i(X)$ denote the number of edges in the graph induced by $X$. Similarly we let $i_L(X)$ (respectively $i_D(X)$) denote the number of length (respectively direction) edges induced by $X$.

%WIP 2015
%DEFN d(X,Y), d(X,Y,Z)
When we take the union of two non-empty vertex sets $X,Y\subset V$, the resulting induced graph $G[X\cup Y]$ contains all of the edges in $G[X]$ and $G[Y]$, but may also contain additional edges which have one endvertex in $X-Y$ and the other in $Y-X$. We denote the number of such edges between $X-Y$ and $Y-X$ by $d(X,Y)$, and extend this notation to three non-empty sets $X,Y,Z\subset V$ by letting $d(X,Y,Z)= d(X, Y-Z)+ d(Y, Z-X)+d(Z, X-Y)$. 

\subsection{Independent Sets and Circuits}

If the edge set of the graph $G$ is independent in the rigidity matroid $\R(G)$, then we say that $G$ is \emph{independent}. Similarly, if the edge set of $G$ is a circuit in $\R(G)$, then we call $G$ a \emph{mixed circuit} when $E(G)$ contains both length and direction edges, and a \emph{pure circuit} otherwise.

%LEMMA INTRO:
By considering the edge density of vertex-induced subgraphs, Servatius and Whiteley \cite{SW_PlaneCAD} found the following characterisation of the rigidity matroid:

\begin{lem}\label{lem:IndependenceSparsity}\cite[Theorem 4]{SW_PlaneCAD}
%NB: Possibly introduce this Lemma as: a characterisation of the rigidity matroid with sparse subgraphs.
A direction-length graph $G=(V;D,L)$ is independent if and only if for all $X \subseteq V$ with $|X|\geq 2$,
\begin{enumerate}
  \item $i(X)\leq 2|X| -2$, and
  \item $i_{D}(X)\leq 2|X|-3$ and $i_{L}(X)\leq 2|X|-3$.
\end{enumerate}
\end{lem}
 
%DEFN: critical
In a direction-length graph $G=(V;D,L)$, let $X\subseteq V$ with $|X|\geq 2$ and $G[X]$ independent. We call $X$ \emph{mixed critical} if $i(X)= 2|X|-2$, or \emph{pure critical} if $i(X)= 2|X|-3$ and either $i_L(X)=0$ or $i_D(X)=0$ (in which case, we call $X$ \emph{direction critical} or \emph{length critical} respectively). We say $X$ is \emph{critical} if it is either mixed or pure critical.

\begin{lem}\label{lem:CriticalUnions}\cite[Lemma 2.4]{JJ_MixedCircuits}
%NB: when using flower arguments, we are considering these as vertex sets in the independent graph V-v
Let $G=(V;D,L)$ be an independent mixed graph.
\begin{enumerate}
  \item If $X$ and $Y$ are mixed critical sets with $X\cap Y\neq\emptyset$, then $X\cap Y$ and $X\cup Y$ are both mixed critical and $d(X,Y)=0$.\label{part:CriticalUnions_BothMixed}
  \item If $X$ and $Y$ are both direction (respectively length) critical sets with $|X\cap Y|\geq 2$, then either \label{part:CriticalUnions_BothPure}
  \begin{enumerate}
    \item $d(X,Y)=0$ and $X\cap Y$, $X\cup Y$ are both direction (respectively length) critical, or
    \item $d(X,Y)=1$, $X\cup Y$ is mixed critical and $i_{D}(X\cup Y)=2|X\cup Y|-3$ (respectively $i_{L}(X\cup Y)=2|X\cup 
Y|-3$).
  \end{enumerate}
  \item If $X$ is mixed critical and $Y$ is pure critical with $|X\cap Y|\geq 2$, then $X\cup Y$ is mixed critical, $X\cap Y$ is pure critical and $d(X,Y)=0$.\label{part:CriticalUnions_MixedPure}
  \item If $X$ is length critical and $Y$ is direction critical with $|X\cap Y|\geq 2$, then $X\cup Y$ is mixed critical, $|X\cap Y|=2$ and $d(X,Y)=0$. 
\end{enumerate}
\end{lem}

\begin{lem}\label{lem:CriticalUnionsMPP}\cite[Lemma 2.5]{JJ_MixedCircuits}
Let $G=(V;D,L)$ be an independent mixed graph with mixed critical set $X$ and pure critical sets $Y$ and $Z$ satisfying $|X\cap Y|=|Y\cap Z|=|X\cap Z| =1$ and $X\cap Y \cap Z = \emptyset$. Then $X\cup Y\cup Z$ is mixed critical and $d(X,Y,Z)=0$.
\end{lem}

The characterisation given in Lemma \ref{lem:IndependenceSparsity} of independent sets in the rigidity matroid as sparse graphs, leads to the following results characterising circuits in the rigidity matroid:

\begin{lem}\label{lem:MixedCircuitGraphSparsityDefn}\cite[Lemma 3.1]{JJ_MixedCircuits}
A direction-length graph $G=(V;D,L)$ is a mixed circuit if and only if
\begin{enumerate}
  \item $|D|+|L|=2|V|-1$,
  \item $i(X)\leq 2|X|-2$ for all $X\subset V$ with $2\leq |X|$, and
  \item $i_{D}(X)\leq 2|X|-3$ and $i_{L}(X)\leq 2|X|-3$ for all $X\subseteq V$ with $|X|\geq 2$.
\end{enumerate}
\end{lem}

\begin{lem}\label{lem:PureCircuitGraphSparsityDefn}\cite[Lemma 3.2]{JJ_MixedCircuits}
A direction-length graph $G=(V;D,L)$ is a pure circuit if and only if
\begin{enumerate}
  \item $|D|+|L|=2|V|-2$ and either $D=\emptyset$ or $L=\emptyset$, and
  \item $i(X)\leq 2|X|-3$ for all $X\subset V$ with $2\leq |X|$.
\end{enumerate}
\end{lem}

%EXAMPLES: minimal circuits
The pure circuit with fewest vertices is a pure $K_4$. The mixed circuits with fewest vertices are denoted by $K_3^+$ and $K_3^-$, and are obtained from a length-pure (respectively direction-pure) $K_3$ by adding two direction (respectively length) edges between distinct pairs of vertices, see Figure \ref{fig:K3+and-}. Servatius and Whiteley \cite{SW_PlaneCAD} characterised rigidity for circuits:

%********************
% K_3^+ and K_3^-
%********************
\begin{figure}
\centering
\includegraphics{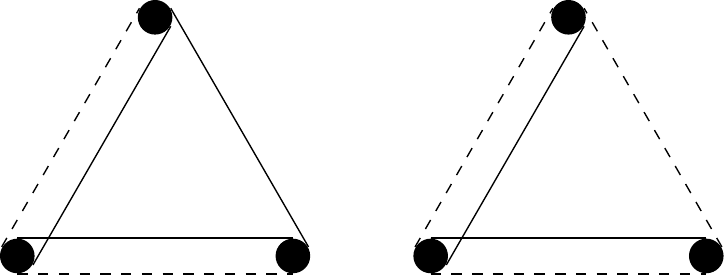}
\caption{The two mixed circuits on three vertices, $K_3^+$ and $K_3^-$.}
\label{fig:K3+and-}
\end{figure}

\begin{lem}\label{lem:RigidCircuits}\cite[Theorems 2, 4]{SW_PlaneCAD} 
Let $G=(V;D,L)$ be a circuit. Then $G$ is (redundantly) rigid if and only if $G$ is mixed.
\end{lem}

The following result implies that the union of intersecting mixed circuits is also rigid: 

\begin{lem}\label{lem:RigidUnion}%TODO - may have seen this stated elsewhere
Let $G=H_{1}\cup H_{2}$ be a mixed graph, with $V(H_{1})\cap V(H_{2})\neq \emptyset$. If $H_{1}$ and $H_{2}$ are rigid then $G$ is rigid.
\end{lem}

\begin{proof}%pg. 23 & Bill's email 15.04.2014
Let $V_i= V(H_i)$ for $i\in \{1,2\}$, and let $p$ be a generic realisation of $G$. Then $(H_1,\left.p\right|_{V_1})$ and $(H_2,\left.p\right|_{V_2})$ are generic realisations of $H_1$ and $H_2$ respectively.

Let $v\in V_1\cap V_2$. Since $H_1$ is rigid and $\left.p\right|_{V_1}$ is a generic realisation of $H_1$, the only motions of $(H_1,\left.p\right|_{V_1})$ are translations. Hence the only motion of $(H_1,\left.p\right|_{V_1})$ which fixes $v$ is a translation of length zero, i.e.\ the motion which fixes all of the vertices in $V_1$.

Similarly, the only motion of $(H_2,\left.p\right|_{V_2})$ which fixes $v$ is the motion which fixes all of the vertices in $V_2$. This implies that the only motion of the entire framework $(G,p)$ which fixes $v$ must fix all the vertices in $G$. Hence $(G,p)$ is rigid.
\end{proof}

\subsection{Graph Connectivity}

%DEFN - connectivity, edge connectivity
Given a direction-length graph $G=(V;D,L)$, a \emph{$k$-vertex-cut} (\emph{$k$-edge-cut}) of $G$ is a set of $k$ vertices (edges) whose removal disconnects $G$. The graph $G$ is called \emph{$k$-connected} if $|V|>k$ and there is no set of vertices of size less than $k$ whose removal disconnects $G$. Similarly, $G$ is \emph{$k$-edge-connected} if all of its edge-cuts have size at least $k$.

%DEFN - separations
A \emph{$k$-separation} of $G$ is a pair of subgraphs $(H_{1},H_{2})$ of $G$ with $|V(H_1)|\geq k+1$ and $|V(H_2)|\geq k+1$ such that $H_{1}\cup H_{2}=G$ and $|V(H_{1})\cap V(H_{2})|=k$. An \emph{edge-disjoint $k$-separation} is a $k$-separation $(H_1,H_2)$ where $H_1$ and $H_2$ are edge-disjoint. It is clear that a graph has a $k$-separation if and only if it has an edge-disjoint $k$-separation on the same vertex sets.

%DEFN - balanced, unbalanced separations
A 2-separation $(H_{1},H_{2})$ of a graph $G$ on a 2-vertex-cut $\{x,y\}$ is \emph{direction-balanced} (respectively \emph{length-balanced}) if both $E(H_{1})-E(\{x,y\})$ and $E(H_{2})-E(\{x,y\})$ contain a direction (resp.\ length) edge. The graph $G$ is \emph{direction-balanced} (resp.\ \emph{length-balanced}) if every 2-separation of $G$ is direction-balanced (resp.\ length-balanced). A graph is \emph{balanced} if it is both direction- and length-balanced, and is \emph{unbalanced} otherwise. 

We have the following results on the connectivity of critical sets and circuits.

\begin{lem}\label{lem:CriticalSet}\cite[Lemma 2.3]{JJ_MixedCircuits}
Let $G=(V;D,L)$ be a mixed graph and let $X\subseteq V$ be a critical set. Then
\begin{enumerate}
  \item $G[X]$ is 2-edge-connected unless $|X|=2$ and $i(X)=1$.\label{part:CriticalSet_2EdgeConnected}
  \item If $(H_{1},H_{2})$ is a 1-separation of $G[X]$ then $X$ is mixed critical and $V(H_{1})$ and $V(H_{2})$ are also mixed critical. \label{part:CriticalSet_1sep}
\end{enumerate}
\end{lem}

\begin{lem}\cite[Lemma 3.3]{JJ_MixedCircuits}\label{lem:Circuits2Connected}
%NB: true for mixed and pure circuits
Let $G$ be a mixed or pure circuit. Then $G$ is 3-edge-connected and 2-connected.
\end{lem}

A trivial, but very useful, consequence of Lemmas \ref{lem:IndependenceSparsity} and \ref{lem:Circuits2Connected} is that all circuits have the same minimum degree:

\begin{cor}\label{cor:CircuitsMinDegree3}
Let $G$ be a mixed or pure circuit. Then $\delta(G)=3$.
\end{cor}

%=======================================================================================================
%=======================================================================================================

\subsection{Operations Preserving Global Rigidity}\label{subsec:Definitions_GlobalRigidityOperations}

The goal of this paper is to characterise global rigidity for all generic direction-length frameworks with a ``connected'' rigidity matroid, by finding a inductive construction of all such graphs which are globally rigid. To this end, we define the following three recursive operations which are known to preserve global rigidity in generic frameworks.

Given $G=(V;D,L)$, an \emph{edge addition} adds a new edge $e$ to $G$ to obtain the graph $G'=G+e$. Whereas the \emph{0-extension operation} adds some new vertex $v$ to $G$, along with two new edges $vx$ and $vy$ for some $x,y\in V$, such that if $x=y$ then these edges are of different type. A 0-extension is \emph{direction pure} if both of the edges added are direction edges.

Finally, the \emph{1-extension operation} deletes some edge $e=xy$ of $G$, and adds a vertex $v$ to $G$, along with edges $vx$, $vy$ and $vz$ for some $z\in V$, such that at least one of these new edges is of the same type as $e$. A graph obtained from $G$ in this manner is denoted $G^v$.

\begin{lem} \cite[Theorems 1.2 and 1.3]{JJ_DL_GlobalOperations} \label{lem:OperationsPreservingGlobalRigidity}
Let $(G,p)$ and $(G',p')$ be generic direction-length frameworks. Suppose that either
\begin{enumerate}
	\item  $(G,p)$ is globally rigid, and $(G',p')$ is obtained from $(G,p)$ by an edge addition or a direction-pure 0-extension, or
	\item $(G,p)$ is globally rigid, $G-e$ is rigid for some edge $e\in E(G)$, and $(G',p')$ is obtained from $(G,p)$ by a 1-extension which deletes the edge $e$.
\end{enumerate}
Then $(G',p')$ is globally rigid.
\end{lem}

By Theorem \ref{thm_JJ:GloballyRigidCircuits} and Lemma \ref{lem:RigidCircuits}, we know that all generic realisations of the smallest circuits, $K_3^+$ and $K_3^-$ are globally rigid and redundantly rigid. Hence all graphs that can be constructed from these two graphs by the operations in Lemma \ref{lem:OperationsPreservingGlobalRigidity} are also globally rigid. %TODO - need to be careful about maintaining redundant rigidity if we want to apply 1-exts later.
However, this paper is concerned with finding globally rigid graphs with ``connected'' rigidity matroids. So we need to identify which of the above operations also preserve  matroid connectivity. This is the focus of the next section.

%#######################################################################################################
%#######################################################################################################
\section{$\M$-Connected Graphs}\label{sec:M-Connected}

First we shall formally define matroid connectivity and $\M$-connectivity, before identifying some properties of $\M$-connected graphs. In particular, in Sections \ref{subsec:MConnOperations} and \ref{subsec:2Sums} we show that a subset of the operations in Lemma \ref{lem:OperationsPreservingGlobalRigidity} which preserve global rigidity, also preserve matroid connectivity.

%DEFN - connected
Given a matroid $\M=(E,\mathcal{I})$, we define a relation on $E$ such that for all $e,f\in E$, the elements $e$ and $f$ are related if either $e=f$ or there exists a circuit $C$ in the matroid, such that $e,f\in C$. It is well-known that this is an equivalence relation (see \cite[Proposition 4.1.2]{Ox_Matroids}). We say that the matroid $\M$ is \emph{connected} if $\M$ has exactly one equivalence class under this relation, and that $\M$ is \emph{trivially} connected if $E$ consists of a single edge. The simplest, non-trivial, connected matroids are circuits. 

%DEFN - M-connected
We define a mixed or pure graph $G=(V;D,L)$ to be \emph{$\M$-connected} if its rigidity matroid, $\R(G)$, is connected but not trivially connected.

\begin{lem}\label{lem:MConnectedGraphsAre2Connected}
Let $G$ be a mixed or pure graph. If $G$ is $\M$-connected then $G$ is 2-connected.
\end{lem}

\begin{proof}
Assume that $G$ has a 1-separation $(H_1,H_2)$ and let $e\in E(H_1)$ and $f\in E(H_2)$. Since $G$ is $\M$-connected, the rigidity matroid of $G$ contains a circuit $C$ such that $e,f\in C$. Lemma \ref{lem:Circuits2Connected} implies that $G[C]$ is 2-connected. But $G[C]$ intersects both $H_1-H_2$ and $H_2-H_1$, which contradicts that $(H_1,H_2)$ is a 1-separation of $G$.
\end{proof}

%=======================================================================================================
%======================================================================================================

\subsection{Ear Decompositions}

The definition of $\M$-connectivity suggests that we can consider the edge set of an $\M$-connected graph
to be an intersecting sequence of circuits. This sequence is called an \emph{ear decomposition} of the edge set, and will allow us to infer properties of the graph from the properties of the circuits in its ear decomposition.

%DEFN - ear decomposition \cite{JJ_LengthConnected}
Let $\M=(E,\mathcal{I})$ be a matroid and $C_{1}, C_{2}, \ldots, C_{m}$ be a non-empty sequence of circuits in $\M$. Let $E_{i}=C_{1}\cup C_{2}\cup\ldots\cup C_{i}$ for all $1 \leq i \leq m$. The sequence $C_{1}, C_{2}, \ldots, C_{m}$ is a \emph{partial ear decomposition} of $\M$ if for all $2\leq i\leq m$
\begin{enumerate}[label=(E\arabic*)]
  \item $C_{i}\cap E_{i-1}\neq \emptyset$, \label{part:E1}
  \item $C_{i}-E_{i-1}\neq\emptyset$, and \label{part:E2}
  \item No circuit $C_{i}'$ satisfying \ref{part:E1} and \ref{part:E2} has $C_{i}'-E_{i-1}\subset C_{i}-E_{i-1}$. \label{part:E3}
\end{enumerate}
The set $\tilde{C}_{i}:=C_{i}-E_{i-1}$ is the \emph{lobe} of the circuit $C_{i}$. A partial ear decomposition with $E_{m}=E$ is called an \emph{ear decomposition of $\M$}.

\begin{lem}\label{lem:EarDecomp}\cite{CH_PortOraclesForMatroids} %NB: mentions algorithm to find decomp. Possibly existence of decomp already proven elsewhere
Let $\M=(E, \mathcal{I})$ be a matroid with $\abs{E}\geq 2$ and rank function $r$. Then
\begin{enumerate}
  \item $\M$ is connected if and only if $\M$ has an ear decomposition. \label{part:EarDecomp_Connected}
  \item If $\M$ is connected then any partial ear decomposition of $\M$ can be extended to an ear decomposition of $\M$. \label{part:EarDecomp_Extending}
  \item If $C_{1},C_{2}, \ldots , C_{m}$ is an ear decomposition of $\M$ then
\[ r(E_{i})-r(E_{i-1})=\abs{\tilde{C}_{i}}-1 \quad \text{for } 2 \leq i \leq m. \] \label{part:EarDecomp_Rank}
\end{enumerate}
\end{lem}
An ear decomposition can contain both pure and mixed circuits. However, many of the properties we wish to infer for $\M$-connected mixed graphs are known to be properties of mixed circuits, but not of pure circuits. So we need to determine when an $\M$-connected graph has an ear decomposition into only mixed circuits.

\begin{lem} \label{lem:MixedMConnectedHasDecompIntoMixedCircuits}
Let $G$ be an $\M$-connected mixed graph. Then $\R(G)$ has an ear decomposition into mixed circuits.
\end{lem}

\begin{proof} %KC
%OUTLINE: start with mixed circuit. Extend to full ear decomposition, Let C_{i} be a pure circuit. Edge transitivity implies there is a circuit which contains edge in lobe and any chosen edge in C_{1}. Call this circuit C'. Ear decomp dfn. implies that C' satisfies rules and can replace C_{i} in decomposition.

Let $l_{1}$ be a length edge and $d_{1}$ a direction edge in $E(G)$. Since $G$ is $\M$-connected, there exists a circuit $C_{1}$ in $\R(G)$ containing both $l_{1}$ and $d_{1}$. Clearly $C_1$ is a mixed circuit. 

If $G$ is a circuit, then $E(G)=C_{1}$ and we are done. So suppose $G$ is not a circuit. Then by Lemma \ref{lem:EarDecomp}\ref{part:EarDecomp_Extending}, it is possible to extend the partial ear decomposition $C_1$ to a full ear decomposition $C_{1},C_{2},\ldots,C_{m}$ of $\R(G)$, for some $m\geq 2$. Suppose this decomposition does not consist solely of mixed circuits, and let $k$ be the least integer such that $C_{k}$ is a pure circuit. By \ref{part:E2} there exists some edge $e_{k}$ in the lobe of $C_{k}$.

Pick $e_{1}\in\{d_{1},l_{1}\}$ of opposite type to $e_{k}$. Since $G[\bigcup_{i=1}^{k}C_{i}]$ is $\M$-connected, there exists some circuit $C_{k}'\subseteq \bigcup_{i=1}^{k}C_{i}$ such that  $e_{1},e_{k}\in C_{k}'$. So $C_{k}'$ is a mixed circuit. Clearly $C_{k}'$ satisfies \ref{part:E1} and \ref{part:E2}. Also, since $C_{k}$ satisfies \ref{part:E3} and $\tilde{C}_{k}'\cap\tilde{C}_{k}\neq\emptyset$  we must have $\tilde{C}_{k}'=\tilde{C}_{k}$, and thus $C_{k}'$ also satisfies \ref{part:E3}. 

Hence $C_{1},\ldots, C_{k-1}, C_{k}',C_{k+1},\ldots,C_{m}$ is an ear decomposition of $\R(G)$ where any pure circuit, $C_{i}$, in the sequence, must have $i>k$. By iteratively applying this argument, we generate an ear decomposition consisting of just mixed circuits.
\end{proof}

This result leads to the following characterisation of rigidity and redundant rigidity for $\M$-connected graphs:

\begin{lem}\label{lem:MixedMConnectedIsRigid}%KC
Let $G=(V;D,L)$ be an $\M$-connected graph. Then $G$ is (redundantly) rigid if and only if $G$ is mixed.
\end{lem}

\begin{proof}
Suppose $G$ is a pure graph. Then any realisation of $G$ can either be continuously rotated (if $G$ is length-pure) or continuously dilated (if $G$ is direction-pure) whilst preserving the edge constraints, so $G$ is neither rigid nor redundantly rigid.

So instead, let $G$ be mixed. Then Lemmas \ref{lem:RigidCircuits} and \ref{lem:MixedMConnectedHasDecompIntoMixedCircuits} imply that $G$ is a union of redundantly rigid mixed circuits $H_1,H_2, \ldots, H_m$ for some $m\geq 1$. Let $e\in E(G)$, and $F_i = H_i-e$ for all $1\leq i \leq m$. Then $G-e$ is the union of the rigid subgraphs $F_1,F_2, \ldots, F_m$. Thus $G-e$ is rigid, by Lemma \ref{lem:RigidUnion}. Hence $G$ is redundantly rigid (and rigid).
\end{proof}

\begin{lem} \label{lem:EarDecompRules}%KC pg.30 (analogue of Lemma 5.2 in (3))
Let $G=(V;D,L)$ be an $\M$-connected mixed graph and let $H_{1},H_{2},\ldots,H_{m}$  be the subgraphs of $G$ induced by the mixed circuits $C_{1}, \ldots, C_{m}$ of an ear decomposition of $\R(G)$, where $m\geq 2$. Let $Y=V(H_{m})-\bigcup_{i=1}^{m-1}V(H_{i})$ and $X=V(H_{m})-Y$. Then:
\begin{enumerate}
  \item $\abs{\tilde{C}_{m}}=2\abs{Y}+1$; \label{part:EarDecompRules_LobeSize}
  \item either $Y=\emptyset$ and $\abs{\tilde{C}_{m}}=1$ or $Y \neq \emptyset$ and every edge $e\in\tilde{C}_{m}$ is incident to $Y$; \label{part:EarDecompRules_EdgeCase}
  \item if $Y \neq \emptyset$, then $X$ is mixed critical in $H_{m}$;\label{part:EarDecompRules_XMixedCritical}
  \item $G[Y]$ is connected; \label{part:EarDecompRules_Connected}
  \item if $G$ is 3-connected, then $\abs{X}\geq 3$. \label{part:EarDecompRules_3ConnectedCase}
\end{enumerate}
\end{lem}

\begin{proof} %KC pg. 30
Let $G_{j}=\bigcup_{i=1}^{j}H_{i}$ and $E_{j}=\bigcup_{i=1}^{j}C_j$. So $E(G_j)=E_j$. Lemma \ref{lem:EarDecomp}\ref{part:EarDecomp_Connected} implies that $G_{m-1}$ is  $\M$-connected. So, by Lemma \ref{lem:MixedMConnectedIsRigid}, both $G_{m-1}$ and $G$ are rigid, which implies that $r(E_{m-1})=2\abs{V-Y}-2$ and $r(E)=2\abs{V}-2$. Thus, by Lemma \ref{lem:EarDecomp}\ref{part:EarDecomp_Rank},
\[ \abs{\tilde{C}_{m}} =  r(E) - r(E_{m-1})+1 = (2\abs{V}-2)-(2\abs{V-Y}-2)+1 = 2\abs{Y}+1.\]
Which gives part \ref{part:EarDecompRules_LobeSize}. Hence, when $Y=\emptyset$ we must have $\abs{\tilde{C}_{m}}=1$. Suppose $Y\neq\emptyset$, and assume that exactly $k$ edges in $E-E_{m-1}$ have both endvertices in $V(G_{m-1})$. Since $H_m$ is a mixed circuit, part \ref{part:EarDecompRules_LobeSize} implies
\[i_{H_m}(X)=|C_m|-|\tilde{C}_m|+k = (2|X\cup Y|-1)-(2|Y|+1)+k = 2|X|+k-2.\]
Since $H_m[X]$ is a proper subgraph of $H_m$, it must be independent. Thus $k=0$, and $X$ is mixed critical in $H_m$, proving \ref{part:EarDecompRules_EdgeCase} and \ref{part:EarDecompRules_XMixedCritical} respectively.
  
  We now consider part \ref{part:EarDecompRules_Connected}. Assume $G[Y]$ is disconnected. Then $G[Y]$ consists of connected components $G[Y_1],G[Y_2],\ldots,G[Y_k]$ for some $k\geq 2$, where $Y_1,Y_2,\ldots,Y_k$ partitions $Y$. Since $H_{m}$ is a circuit, $H_m[Y_i]$ is sparse for all $1\leq i \leq k$. Hence, for each component of $Y$,
\[ i_{H_{m}}(X\cup Y_i)-i_{H_{m}}(X)\leq (2|X\cup Y_i|-2) - (2|X|-2)=2|Y_i|, \]
which implies that
\[ |\tilde{C}_m|= \sum_{i=1}^{k}(i_{H_{m}}(X\cup Y_i)-i_{H_{m}}(X))\leq \sum_{i=1}^{k}2|Y_{i}|= 2|Y| \]
contradicting part \ref{part:EarDecompRules_LobeSize}.
  
  Finally, we consider part \ref{part:EarDecompRules_3ConnectedCase}. Suppose $G$ is 3-connected. If $Y \neq\emptyset$, then $X$ is a separator of $G$ and so $|X|\geq 3$. If $Y=\emptyset$, then $X$ is the vertex set of a mixed circuit. The smallest mixed circuits, $K_{3}^{+}$ and $K_{3}^{-}$, have 3 vertices. Hence $|X|\geq 3$.
\end{proof}

%=======================================================================================================
%======================================================================================================

\subsection{Operations Preserving $\M$-connectivity}\label{subsec:MConnOperations}

Here we show that two of the operations in Section \ref{subsec:Definitions_GlobalRigidityOperations} which preserve global rigidity: edge additions and 1-extensions, also preserve $\M$-connectivity for mixed graphs. We start with edge additions:

\begin{lem}\label{lem:MixedMConnectedPreservedByEdgeAdditions}
Let $G=(V;D,L)$ be an $\M$-connected mixed graph and let $G'$ be obtained from $G$ by an edge addition. Then $G'$ is $\M$-connected and mixed.
\end{lem}

\begin{proof}
Since $E(G)\subset E(G')$, $G'$ is mixed. Denote the edge added in the edge addition by $e$. By Lemma \ref{lem:MixedMConnectedIsRigid}, both $G$ and $G'$ are rigid on the same vertex set. Hence $r(\R(G))=2|V|-2=r(\R(G'))$. Let $B$ be a maximal independent set in $E(G)$. Then $r_{G}(B)=r(\R(G))=r(\R(G'))$, so $B$ is also maximally independent in $G'$. Hence $B+e$ is dependent in $G'$, which implies $\R(G')$ contains a circuit $C$ such that $e\in C\subseteq B+e$. Since $|C|\geq i(K_3^+)=5$, we know that $C\cap E(G)\neq \emptyset$. Thus $\R(G')$ is non-trivially connected, and so $G'$ is $\M$-connected.
\end{proof}

%DEFN - 1-extension (contd)
Showing that 1-extensions preserve $\M$-connectivity requires more work. We say that a 1-extension is \emph{pure} if all the edges added are of the same type as the edge removed, otherwise it is \emph{mixed}. We already know that 1-extensions and edge additions preserve $\M$-connectivity in the following cases:

\begin{lem}\cite[Lemma 3.9]{JJ_LengthConnected}\label{lem:PureMConnectedPreservedBy1extensions}
Let $G=(V,E)$ be an $\M$-connected pure graph and let $G'$ be obtained from $G$ by either a pure 1-extension or an edge addition, where in both cases the edges added are of the same type as $G$. Then $G'$ is pure and $\M$-connected.
\end{lem}

\begin{lem}\cite[Lemma 3.6]{JJ_MixedCircuits}\label{lem:MixedCircuitPreservedBy1extensions}
Let $G$ be a mixed circuit and $G'$ be a 1-extension of $G$. Then $G'$ is a mixed circuit.
\end{lem}

We shall extend these results to all $\M$-connected mixed graphs. But to do this, we need the following lemma, which gives us a way of transferring results for pure 1-extensions to mixed 1-extensions:

\begin{lem}\label{lem:MixedCircuitCanChangePureVertexToMixedVertex}
%Original - any pure node in a mixed circuit can have the type of two edges changed
%2014.10.06 - extended to any pure vertex
Let $G=(V;D,L)$ be a mixed circuit and let $v$ be a pure vertex in $G$. Let $G'$ be the graph obtained from $G$ by changing the type of at most two of the edges incident with $v$. Then $G'$ is a mixed circuit. %NB - any vertex! Not necessarily a node!
\end{lem}

\begin{proof}
Let $\{vx,vy\}$ be the set of edges whose type was changed. By Corollary \ref{cor:CircuitsMinDegree3},  $d_G(v)\geq 3$, so at least one of the edges terminating at $v$ was not changed. Hence $E(G')$ is mixed. Since $\abs{E(G')}=\abs{E(G)}=2\abs{V}-1$, we know that $E(G')$ is dependent, and so there exists some set of edges $C\subseteq E(G')$ which is a circuit in the rigidity matroid.

If neither $vx$ nor $vy$ is contained in $C$, then $C\subset E(G)$, which contradicts that $G$ is a circuit. Hence $C$ must contain at least one of these changed edges, say $vx$. But $G'[C]$ is a circuit, so Corollary \ref{cor:CircuitsMinDegree3} implies that $C$ contains at least 3 edges incident with $v$. By the construction of $G'$ from $G$, this implies that $C$ contains an edge of opposite type to $vx$. Hence $C$ is mixed. If $C\neq E(G')$ then the corresponding set of edges in $G$ breaks the circuit sparsity condition for $G$. Thus $C=E(G')$, and $G'$ is a mixed circuit.
\end{proof}

We now show that 1-extensions preserve $\M$-connectivity for mixed graphs:

\begin{lem}\label{lem:MixedMConnectedPreservedBy1extensions}
Let $G=(V;D,L)$ be an $\M$-connected mixed graph and let $G'$ be obtained from $G$ by a 1-extension. Then $G'$ is mixed and $\M$-connected.
\end{lem}

\begin{proof}
%NB: we can't just use that "a 1-extension of a mixed circuit is a mixed circuit" \ref{lem:MixedCircuitPreservedBy1extensions} because $xy\in C$ does not imply that $z\in C$. So we would not actually be doing a 1-extension of $C$. Methods to ensure $z\in C$ don't guarantee that $C$ is mixed - it might be pure (there may be a way around this problem). If $C$ is pure, a mixed 1-extension would give a mixed critical set, not a circuit.

Let the 1-extension used to obtain $G'$ from $G$ add the vertex $v$ with neighbourhood $\{x,y,z\}$ (where potentially $x=z$) whilst removing an $xy$-edge $e$. We shall use the transitivity of matroid connectivity to prove that $G'$ is $\M$-connected, by showing that given some $e_1\in E(G')$, we can find a circuit containing both $e_1$ and $e_2$ for all $e_2\in E(G')-e_1$.

Suppose $x=z$. Pick some edge $g\in E(G)-e$ of opposite type to $e$. Since $G$ is $\M$-connected, for all $f\in E(G)-g$, there is a circuit $C$ in $\R(G)$ such that $f,g\in C$. If $e\not\in C$ then $C\subset E(G')$ and we are done. So instead assume $e\in C$. Then $C$ is mixed. Since $G[C]$ contains $N_{G'}(v)=\{x,y\}$, the 1-extension which builds $G'$ from $G$ induces a 1-extension $G'[C']$ of $G[C]$. By Lemma \ref{lem:MixedCircuitPreservedBy1extensions}, $G'[C']$ is a mixed circuit and contains the edges $g,vx,vy$ and $vz$, as well as the edge $f$ (when $f\neq e$). Hence $G'$ is $\M$-connected.

So instead suppose $x,y$ and $z$ are distinct. There are two cases to consider: when $v$ is added to $G$ by a pure 1-extension, and when it is added by a mixed 1-extension. 

First, suppose that $v$ is added by a pure 1-extension. Pick some edge $g\in E(G)-e$ which has $z$ as an endvertex. Since $G$ is $\M$-connected, for all $f\in E(G)-g$ there is a circuit $C$ in the rigidity matroid of $G$ such that $f,g\in C$. If $e\not\in C$ then $C\subset E(G')$ and we are done. So suppose that $e\in C$. Since $G[C]$ contains both $e$ and the vertices $x,y$ and $z$, the 1-extension used to form $G'$ from $G$, is also a pure 1-extension, $G'[C']$, of $G[C]$. Hence, using Lemma \ref{lem:MixedCircuitPreservedBy1extensions} if $C$ is mixed, or Lemma \ref{lem:PureMConnectedPreservedBy1extensions} if $C$ is pure; $G'[C']$, is a circuit and contains the edges $vx,vy$ and $vz$ as well as the edges $f$ (when $f\neq e$) and $g$ as required. Hence $G'$ is $\M$-connected.

It remains to show that the claim holds when $v$ is added by a mixed 1-extension. Let the graph obtained by this mixed 1-extension be denoted by $G''$. This graph, $G''$, can be obtained from the corresponding pure 1-extension, $G'$, above, by changing the type of at most two of the edges in $\{vx,vy,vz\}$. Let $C'$ be a mixed circuit in $E(G')$ and suppose $G'[C']$ does not contain $v$. Then $C'\subseteq E(G'')-\{vx,vy,vz\}$ and we are done. Otherwise $G'[C']$ contains $v$, and since circuits have minimum degree 3, this implies $vx,vy,vz\in C'$. We can obtain the corresponding edge set $C''\subseteq E(G'')$ from $C'$ by changing the type of at most two of the edges in $\{vx,vy,vz\}$, as determined above. By Lemma \ref{lem:MixedCircuitCanChangePureVertexToMixedVertex}, $C''$ is a mixed circuit. Since $G'$ is $\M$-connected, and every mixed circuit $C'$ in $G'$ has a corresponding mixed circuit $C''$ in $G''$, Lemma \ref{lem:MixedMConnectedHasDecompIntoMixedCircuits} implies that $G''$ is $\M$-connected.
\end{proof}

%=======================================================================================================
%=======================================================================================================

\subsection{2-Sums}\label{subsec:2Sums}

We have seen that edge additions and 1-extensions preserve $\M$-connectivity. However, any graph we construct from $K_3^+$ or $K_3^-$ with just these operations will be 3-connected, and 3-connectivity is not a requirement of $\M$-connectivity. In Lemma \ref{lem:OperationsPreservingGlobalRigidity} we mentioned a third operation which preserves global rigidity: direction-pure 0-extensions. These add a vertex of degree 2 to our graph, so do not preserve 3-connectivity. But nor do they preserve $\M$-connectivity, as $\M$-connected graphs have minimum degree 3.

Here we introduce a new operation, 2-sums, which preserve $\M$-connectivity whilst allowing us to construct graphs which are not 3-connected. We then show that a 2-sum with a direction-pure $K_4$ is an operation which preserves both global rigidity and $\M$-connectivity.

Let $G_1=(V_1;D_1,L_1)$ be a mixed graph and $G_2=(V_2;P)$ be a direction- (respectively length-) pure graph with $V_1\cap V_2 = \{x,y\}$ and $D_1\cap P=\{xy\}$ (respectively $L_1\cap P=\{xy\}$). The graph $G=(V;D,L)$ is a \emph{2-sum of $G_1$ and $G_2$}, written $G=G_1\oplus_2 G_2$, if $V=V_1\cup V_2$, $D=(D_1\cup P)-\{xy\}$ and $L=L_1$ (respectively $D=D_1$ and $L=(L_1\cup P)-\{xy\}$).

Let $G=(V;D,L)$ be a mixed or pure graph with an edge-disjoint 2-separation $(H_1,H_2)$ on the 2-vertex-cut $\{x,y\}$ where $H_2$ is pure and $G$ does not contain an $xy$-edge of the same type as $H_2$. Then the \emph{2-cleave} of $G$ across the pair $\{x,y\}$ adds the edge $xy$ of the same type as $H_2$ to both $H_1$ and $H_2$ to form the graphs $G_1$ and $G_2$ respectively, such that $G=G_1 \oplus_2 G_2$.

%DEFN - neighbourhood N(), end
When we make a 2-cleave, we often want to ensure we are removing the fewest possible vertices from our graph. To aid our description of this, we introduce the following notation. Given a graph $G=(V;D,L)$ and a set $X\subset V$, the \emph{neighbourhood} of $X$ in $G$, is given by $N_G(X)=\{v\in V-X : xv\in E(G) \text{ for some } x\in X\}$. When it is clear which graph we are referring to, we simply write $N(X)$. We call $X\subset V$ an \emph{end} of $G$ if $|N(X)|=2$, $V-(X\cup N(X))\neq\emptyset$, and for all non-empty $X'\subset X$ we have $|N(X')|\geq 3$.

It is already known that 2-sums and 2-cleaves preserve $\M$-connectivity for pure and mixed circuits:

\begin{lem} \label{lem:2SumPureCircuit}\cite[Lemmas 4.1, 4.2]{BJ_LengthCircuits} %Lemmas 4.1 and 4.2
Let $G$ be a pure graph.
\begin{enumerate}
	\item Suppose $G$ is the 2-sum of two pure graphs $G_1$ and $G_{2}$. If $G_{1}$ and $G_{2}$ are circuits then $G$ is a pure circuit.\label{part:2SumPureCircuit_2Sum}
	\item Suppose $G$ is a pure circuit with an edge-disjoint 2-separation $(H_1,H_2)$ on the 2-vertex-cut $\{x,y\}$. Then $d_G(x),d_G(y)\geq 4$ and $xy$ is not an edge of $G$. In addition, if we 2-cleave $G$ across $\{x,y\}$, we obtain graphs $G_1$ and $G_2$ from $H_1$ and $H_2$ respectively, such that $G_1$ and $G_2$ are both pure circuits.\label{part:2SumPureCircuit_Cleave}
\end{enumerate}
\end{lem}

\begin{lem} \label{lem:2SumMixedCircuit}\cite[Lemma 3.7]{JJ_MixedCircuits} %Lemma 3.7
Let $G$ be a mixed graph.
\begin{enumerate}
	\item Suppose $G$ is the 2-sum of two graphs $G_1$ and $G_{2}$. If $G_{1}$ is a mixed circuit and $G_{2}$ is a pure circuit then $G$ is a mixed circuit.\label{part:2SumMixedCircuit_2Sum}
	\item Suppose $G$ is a mixed circuit and has an edge-disjoint 2-separation $(H_1,H_2)$ on the 2-vertex-cut $\{x,y\}$ with $H_2$ pure. Then $d_G(x), d_G(y)\geq 4$ and $G$ does not contain an $xy$-edge of the same type as $H_2$. In addition, if we 2-cleave $G$ across $\{x,y\}$, we obtain the graphs $G_1$ and $G_2$ from $H_1$ and $H_2$ respectively, such that $G_1$ is a mixed circuit and $G_2$ is a pure circuit. \label{part:2SumMixedCircuit_Cleave}
\end{enumerate}
\end{lem}

We can extend this result on mixed circuits to mixed $\M$-connected graphs:

\begin{lem} \label{lem:2SumMConn} %KC equiv. to \cite{JJ_MixedCircuits} Lemma 3.7 on circuits
Let $G$ be a mixed graph.
\begin{enumerate}
  \item Suppose $G$ is the 2-sum of two graphs $G_{1}$ and $G_{2}$. If $G_{1}$ is mixed, $G_{2}$ is pure and both are $\M$-connected then $G$ is an $\M$-connected mixed graph. \label{part:2SumMConn_Sum}
  \item Suppose $G$ is an $\M$-connected mixed graph and has an edge-disjoint 2-separation $(H_{1},H_{2})$ on 2-vertex-cut $\{x,y\}$ with $H_2$ pure. Then $d_G(x)\geq 4$ and $d_G(y)\geq 4$. Further, if $G$ contains an $xy$-edge of the same type as $H_2$ then $G-xy$ is $\M$-connected. Otherwise, we can 2-cleave $G$ across $\{x,y\}$ to form the graphs $G_1$ and $G_2$ from $H_1$ and $H_2$ respectively, where $G_1$ is an $\M$-connected mixed graph and $G_2$ is an $\M$-connected pure graph.\label{part:2SumMConn_Cleave}
\end{enumerate}
\end{lem}

%TODO WIP Sep 2015 - case b, 
% (1) Need to PROVE d(x),d(y)>= 4. 
% (2) Also this is true regardless of whether xy is an edge or not. REWORD Lemma

\begin{proof} %pg. 83
First we prove part \ref{part:2SumMConn_Sum}. Let $e$ be the edge removed from both $G_1$ and $G_2$ by the 2-sum operation and let $f_i\in E(G_i)-\{e\}$ for $i\in\{1,2\}$. Since $G_1$ and $G_2$ are both $\M$-connected, their rigidity matroids contain circuits $C_1$ and $C_2$ respectively such that $e,f_i\in C_i$ for $i\in\{1,2\}$, $C_2$ is pure and $C_1$ is either mixed, or pure of the same type as $C_2$ (since it contains $e$). Thus by Lemma \ref{lem:2SumPureCircuit}\ref{part:2SumPureCircuit_2Sum} or \ref{lem:2SumMixedCircuit}\ref{part:2SumMixedCircuit_2Sum} as applicable, $C_1\oplus_2 C_2$ is a circuit in $G$ containing both $f_1$ and $f_2$. Hence, by the transitivity of matroid connectivity, $G$ is $\M$-connected.
	
	We shall now prove part \ref{part:2SumMConn_Cleave}. Assume $G$ contains the $xy$-edge $e$ of the same type as $H_2$. Let $f_i\in E(H_i)-\{e\}$ for $i\in\{1,2\}$. Since $G$ is $\M$-connected, $\R(G)$ contains a circuit $C\subseteq E$ such that $f_1,f_2\in C$. By Lemma \ref{lem:2SumPureCircuit}\ref{part:2SumPureCircuit_Cleave} or \ref{lem:2SumMixedCircuit}\ref{part:2SumMixedCircuit_Cleave} as relevant, $e\not\in C$. Hence, by the transitivity of matroid connectivity, $G-xy$ is $\M$-connected. 	Further, by Lemma \ref{lem:Circuits2Connected}, vertices $x$ and $y$ have degree at least 3 in $G[C]$, but both $x$ and $y$ are also endvertices of $e$ in $G$. Hence $d_G(X),d_G(y)\geq 4$.
		
	So instead assume $e\not\in E(G)$. Then $e$ can be added to both $H_1$ and $H_2$ to form $G_1$ and $G_2$ respectively. Let $f_1\in E(H_1)$ and $f_2\in E(H_2)$ as before. Since $G$ is $\M$-connected, $\R(G)$ contains a circuit $C$ such that $f_1,f_2\in C$. Since circuits are 2-connected, both $x$ and $y$ are vertices in $G[C]$. Thus by Lemma \ref{lem:2SumPureCircuit}\ref{part:2SumPureCircuit_Cleave} or \ref{lem:2SumMixedCircuit}\ref{part:2SumMixedCircuit_Cleave} as relevant, $C_2=(C\cap E(H_2))+e$ is a pure circuit in $\R(G_2)$ and $C_1=(C\cap E(H_1))+e$ is a pure (resp.\ mixed) circuit in $\R(G_1)$ when $C\cap E(H_1)$ is pure (resp.\ mixed). So, by the transitivity of circuits, $G_1$ and $G_2$ are both $\M$-connected.
	
	Finally, by \ref{cor:CircuitsMinDegree3}, we have $d_{G_i}(x)\geq 3$ and $d_{G_i}(y)\geq 3$, for $i\in \{1,2\}$, since $\M$-connected graphs are the union of circuits. But $xy$ is an edge in both $G_1$ and $G_2$. Hence $d_G(x)= (d_{G_1}(x)-1) + (d_{G_2}(x)-1) \geq 4$, and similarly, $d_G(y)\geq 4$.
\end{proof}

Lemma \ref{lem:2SumMConn} tells us that the 2-sum of an $\M$-connected mixed graph $G$ with a direction-pure $K_4$ on some direction edge $xy$, will also be mixed and $\M$-connected, but we want this move to also preserve global rigidity. This 2-sum is equivalent to first performing a 0-extension on $x$ and $y$, followed by a direction-pure 1-extension on $xy$. Since $G$ is a circuit, Lemma \ref{lem:RigidCircuits} implies that deleting the edge $xy$ preserves rigidity. Hence, by Lemma \ref{lem:OperationsPreservingGlobalRigidity}, both of these operations preserve global rigidity, which gives us the result we seek:

\begin{lem}\label{lem:2SumDPureK4PreservesGlobalRigidity}
Let $G$ be an $\M$-connected mixed graph which is globally rigid. Let $G'$ be obtained from $G$ by a 2-sum with a direction-pure $K_4$. Then $G'$ is mixed, $\M$-connected and globally rigid.
\end{lem}

\subsection{Crossing 2-separators}

Lemma \ref{lem:2SumDPureK4PreservesGlobalRigidity} tells us that 2-sums with direction-pure $K_4$'s preserve global rigidity, but we have not considered when 2-cleaves preserve global rigidity. In Lemma \ref{lem_JJK:NecessaryConditionsForGlobalRigidity}, we saw that being direction-balanced is a necessary condition for global rigidity, so we need to identify when 2-cleaves preserve being direction-balanced. To do this, we first show that if an $\M$-connected mixed graph is unbalanced, then it has no ``crossing 2-separators''.

%DEFN - crossing 2-separator (\cite{JJ_MixedCircuits} pg. 10)
Let $G$ be a mixed or pure graph with two 2-separations $(H_1,H_2)$ and $(H_1',H_2')$ on 2-vertex-cuts $\{x,y\}$ and $\{x',y'\}$ respectively. If $x$ and $y$ are in different components of $G-\{x',y'\}$ then we say that $\{x,y\}$ \emph{crosses} $\{x',y'\}$. It is clear that if $\{x,y\}$ crosses $\{x',y'\}$, then $\{x',y'\}$ crosses $\{x,y\}$. Thus we can refer to $\{x,y\}$ and $\{x',y'\}$ as \emph{crossing 2-separators}, and we say that the 2-separations $(H_1,H_2)$ and $(H_1',H_2')$ \emph{cross}. Further, if $\{x,y\}$ and $\{x',y'\}$ cross, then neither $xy$ nor $x'y'$ are edges in $G$, so the 2-separations $(H_1,H_2)$ and $(H_1',H_2')$ are both edge-disjoint. See Figure \ref{fig:Crossing2Seps}.

%*************************
% Crossing 2-separator
%*************************
\begin{figure}
\centering
\includegraphics{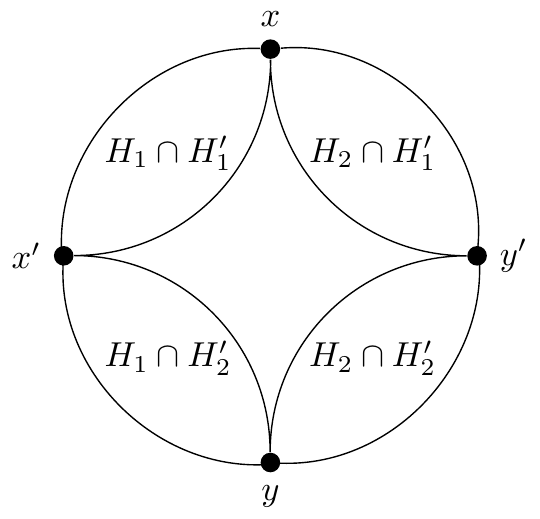}
\caption{Two crossing 2-separations of a graph: $(H_1,H_2)$ on 2-vertex-cut $\{x,y\}$, and $(H_1',H_2')$ on $\{x',y'\}$.}
\label{fig:Crossing2Seps}
\end{figure}

\begin{lem}\label{lem:UnbalancedMixedMConnHasNoCrossing2Sep}%equiv to \cite{JJ_MixedCircuits} Lemma 3.8 (2-seps of mixed circuits where H_2 is pure)
Let $G$ be an $\M$-connected mixed graph and let $(H_{1},H_{2})$ and $(H'_{1},H'_{2})$ be two 2-separations of $G$. If $H_{2}$ is pure then $(H_{1},H_{2})$ and $(H'_{1},H'_{2})$ do not cross.
\end{lem}

\begin{proof}%pg. 88. Borrowed heavily from proof of Lemma 3.8 in \citep{JJ_MixedCircuits}
%OUTLINE
%(1) Assume separations cross.
%(2) C a mixed circuit from D-edge in H_1 to L-edge in H_2
%(3) C 2-connected => C intersects all four quadrants of G
%(4) sparsity conditions => XX
Assume, for a contradiction, that $(H_1,H_2)$ and $(H_1',H_2')$ cross, and let their 2-vertex-cuts be $\{x,y\}$ and $\{x',y'\}$ respectively. Since these 2-vertex-cuts cross, neither $xy$ nor $x'y'$ are edges in $G$.

Let $e_2$ be an edge in $E(H_2)$, and $e_1$ be an edge of opposite type in $E(H_1)$. Since $G$ is $\M$-connected, there is a mixed circuit $C$ in $\R(G)$ such that $e_1,e_2\in C$. But by Lemma \ref{lem:Circuits2Connected}, $G[C]$ is 2-connected which implies that $x,y,x'$ and $y'$ are all vertices in $G[C]$. Hence
\[ |C|= |C \cap E(H_1)|+ |C \cap E(H_2)\cap E(H_1')|+ |C \cap E(H_2)\cap E(H_2')|.  \]
Let $J=V(G[C])$, and let $V_i=V(H_i)$ and $V_i'=V(H_i')$ for $i\in\{1,2\}$. Since $H_2$ is pure, the sparsity conditions for the mixed circuit $C$ give
\begin{align*}
	|C| &\leq (2|J\cap V_1|-2) + (2|J\cap V_2\cap V_1'|-3) + (2|J\cap V_2\cap V_2'|-3) \\
			&= 2(|J|+|\{x,y,y'\}|)-8 = 2|J| -2
\end{align*}
which contradicts the edge count for a circuit. Hence $(H_1,H_2)$ and $(H_1',H_2')$ do not cross.
\end{proof}

\begin{lem}\label{lem:MConnDBal2SepPreservesDBal}
Let $G=G_1\oplus_2 G_2$ be a mixed graph with $G_2$ direction-pure. Then $G_1$ is direction-balanced if and only if $G$ is direction-balanced.
\end{lem}

\begin{proof}
The forwards direction is trivial, so we shall only prove the converse. Let $V(G_1)\cap V(G_2)=\{x,y\}$. Assume that $G_1$ is not direction-balanced. Then there is some end $X$ of $G_1$, such that no direction edges in $G_1$ have an endvertex in $X$. If $X$ is also an end of $G$, then this contradicts that $G$ is direction-balanced. Hence $X\cap\{x,y\}\neq\emptyset$. But $xy$ is a direction edge in $G_1$, which contradicts our original assumption.
\end{proof}

\section{Admissible Nodes in Mixed Circuits}\label{sec:Circuits}

In the previous section we showed that performing  1-extensions, edge additions and 2-sums with pure $K_4$'s preserves $\M$-connectivity. So to obtain our recursive construction of $\M$-connected mixed graphs, we need to show that every $\M$-connected mixed graph (other than $K_3^+$ and $K_3^-$) can be obtained from a smaller $\M$-connected mixed graph by one of these operations.

Of these three operations, it is most difficult to identify when an $\M$-connected mixed graph is a 1-extension of another $\M$-connected mixed graph. In this section, we consider the simplest case of an $\M$-connected graph: a circuit. We review and extend Jackson and Jord\'{a}n's methods in \cite{JJ_MixedCircuits} for identifying when a circuit is a 1-extension of another circuit. We will then extend these results to all $\M$-connected graphs in Section \ref{sec:Admissible}.

%DEFN - node, node subgraph,
Given a mixed or pure graph $G=(V;D,L)$, any vertex of degree three in $G$ is called a \emph{node}, and the set of all such vertices is denoted by $V_{3}$. We call $G[V_{3}]$ the \emph{node subgraph of $G$}. A node of $G$ with degree at most one (exactly two, exactly three) in $G[V_{3}]$ is called a \emph{leaf node} (\emph{series node}, \emph{branching node} respectively). 

\begin{lem}\label{lem:NodeForest}\cite[Lemma 3.4]{JJ_MixedCircuits}
Let $G=(V;D,L)$ be a mixed circuit. Then $G[V_{3}]$ is a forest.
\end{lem}

\begin{lem}\label{lem:MixedCircuitExclCriticalSetContainsNode}\cite[Lemma 3.5]{JJ_MixedCircuits}
Let $G=(V;D,L)$ be a mixed circuit and $X\subset V$ be a mixed critical set. Then $G$ has a node in $V-X$.
\end{lem}

%DEFN: 1-reduction
Given any node $v$ in $G$, the \emph{1-reduction operation} at $v$ on edges $vx$ and $vy$ deletes $v$ and all edges incident to $v$ and adds a new edge $xy$ with the proviso that if $v$ is a pure node then $xy$ must be of the same type as $v$. The graph obtained by this operation is denoted by $G_{v}^{xy}$ and is called a \emph{1-reduction} of $G$. The 1-reduction operation is the inverse of the 1-extension operation.

%DEFN: Admissible
If $G$ is an $\M$-connected mixed (pure) graph, then a 1-reduction is called \emph{admissible} if the resulting graph is mixed (pure) and $\M$-connected. A node $v$ of $G$ is called \emph{admissible} if there is an admissible 1-reduction at $v$, and is \emph{non-admissible} otherwise. In this section, we are considering the special case where $G$ is a mixed circuit. For mixed (pure) circuits, a 1-reduction is admissible if it results in a smaller mixed (pure) circuit.

%MOTIVATION
Let $G=(V;D,L)$ be a mixed circuit with a 1-reduction at $v$ onto the edge $xy$. Assume $G$ contains some critical set $Z\subset V-v$ such that $x,y\in Z$, and $Z$ is either mixed, or pure of the same type as the $xy$-edge added in the 1-reduction. Then $G[Z]+xy$ is dependent in $G_{v}^{xy}$. Since $Z\neq V-v$, this implies that $G_{v}^{xy}$ is not a circuit. So the existence of the critical set $Z$ prevents this 1-reduction from being admissible. In fact, Jackson and Jord{\'a}n \cite{JJ_MixedCircuits} have shown that we can determine the admissibility of nodes in mixed circuits solely by the absence of such sets. However, we need to avoid different combinations of critical sets, depending on whether the node is pure or mixed. 

%DEFN -strong flower, weak flower, clover.
Let $G$ be a mixed circuit, and $v$ be a node of $G$ with three distinct neighbours: $r,s$ and $t$. Let $R,S$ and $T$ be critical sets in $G-v$ with $\{s,t\}\subseteq R\subseteq V-\{v,r\}$, $\{r,t\}\subseteq S\subseteq V-\{v,s\}$ and $\{r,s\}\subseteq T\subseteq V-\{v,t\}$ such that either
\begin{enumerate}
	\item $R$, $S$ and $T$ are all mixed critical, \label{defn:StrongFlower}
	\item $v$ is a pure node, $R$ and $S$ are both mixed critical, and $T$ is pure of the same type as $v$, or \label{defn:WeakFlower}
	\item  $v$ is a pure node, $R$ is mixed critical, and $S$ and $T$ are pure of the same type as $v$.\label{defn:Clover}
\end{enumerate}
We say that the triple $(R,S,T)$ is a \emph{strong flower} on $v$ if it satisfies \ref{defn:StrongFlower}, or a \emph{weak flower} on $v$ if it satisfies \ref{defn:WeakFlower}; and that $(R,S,T)$ is a \emph{flower} if it is either a strong or a weak flower (see Figure \ref{fig:Flower}). If instead $(R,S,T)$ satisfies \ref{defn:Clover}, then we say $(R,S,T)$ is a \emph{clover} on $v$ (see Figure \ref{fig:Clover}). %TODO - possibly rename these as flower type 1,2 and 3 (where # refers to # of mixed critical sets). Remember that clovers look very diffferent to the other examples (clover => unbalanced 2-sep => no flower).
Flowers and clovers satisfy the following, very restrictive, properties:

%********************
% Flower
%********************
\begin{figure}
\centering
\includegraphics{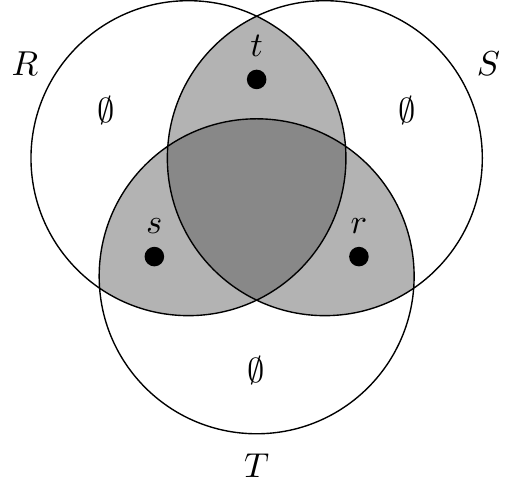}
\caption{Flower formed from vertex sets $R,S$ and $T$.}
\label{fig:Flower}
\end{figure}

\begin{lem}\label{lem:Flowers}\cite[Lemma 4.2]{JJ_MixedCircuits}
Let $G=(V;D,L)$ be a mixed circuit and let $v$ be a node of $G$. Suppose there exists a strong or weak  flower $(R,S,T)$ on $v$, and let $W^*=(V-v)-W$ for all $W\in \{R,S,T\}$. Then
\begin{enumerate}
	\item $R\cup S=S\cup T=R\cup T=V-v$,
	\item $R\cap S\cap T\neq\emptyset$,
	\item $d(R,S)=d(S,T)=d(R,T)=0$, and
	\item $\{R^{*},S^{*},T^{*},R\cap S\cap T\}$ is a partition of $V-v$.
\end{enumerate}
\end{lem}
 
\begin{lem}\label{lem:Clover}
Let $G=(V;D,L)$ be a mixed circuit and let $v$ be a pure node of $G$ with neighbourhood $\{r,s,t\}$. Suppose $(R,S,T)$ is a clover on $v$ with $R$ mixed critical. Then
\begin{enumerate}
	\item $|R|\geq 3$,\label{part:Clover_R} %G mixed so contains at least two edges of opposite type to S,T. Both these edges must occur in $R$. Thus |R|\geq 3.
	\item $R\cup S\cup T =V-v$,\label{part:Clover_V-v} 
	\item $|R\cap S|=|S\cap T|=|R\cap T|=1$,\label{part:Clover_PairsIntersectInSingleVertex} 
	\item $R\cap S\cap T=\emptyset$,\label{part:Clover_TripleIntersectionEmpty}
	\item $d(R,S,T)=0$, and \label{part:Clover_RSTCoverAllEdges}%NB: this is not the same as d(R,S)=d(S,T)=d(R,T)=0. e.g. We could have RS-edges but these must be contained in G[T].
	\item $(G[R],G[S\cup T+ v]-E(R))$ is an unbalanced, edge-disjoint 2-separation of $G$ on 2-vertex cut $\{s,t\}$ with $G[S\cup T+v]-E(R)$ pure.\label{part:Clover_Unbalanced2Separation}
\end{enumerate}
\end{lem}

\begin{proof} %\cite[Proof of Lemma 4.3]{JJ_MixedCircuits}
This proof closely follows that of Lemma 4.3 in \cite{JJ_MixedCircuits}. Let $v$ be a node of type $P\in\{D,L\}$. First assume $|S\cap T|\geq 2$. Then Lemma \ref{lem:CriticalUnions}\ref{part:CriticalUnions_BothPure} implies $i_P(S\cup T)=2|S\cup T|-3$. Since $N_G(v)\subseteq S\cup T$, we know $G[S\cup T +v]$ contains a pure circuit of type $P$, contradicting that $G$ is a mixed circuit. Hence $|S\cap T|=1$, and more specifically, $S\cap T = \{r\}$.

Instead, assume $|R\cap S|\geq 2$. Lemma \ref{lem:CriticalUnions}\ref{part:CriticalUnions_MixedPure} implies $R\cup S$ is mixed critical with $d(R,S)=0$. Since $N_G(v)\subseteq R\cup S$, we know $G[R\cup S +v]$ contains a circuit. Hence $R\cup S = V-v$. Since $r,s\in T$ and $S\cap T = \{r\}$, we have that $T$ intersects both $R-S$  and $S-R$, but does not intersect $R\cap S$. But nor are there any edges from $R-S$ to $S-R$, since $d(R,S)=0$. This implies $T$ is disconnected in $G$, contradicting Lemma \ref{lem:CriticalSet}\ref{part:CriticalSet_2EdgeConnected}. So our assumption is false, and $R\cap S=\{t\}$. A similar argument gives $R\cap T=\{s\}$. Thus proving parts \ref{part:Clover_PairsIntersectInSingleVertex} and \ref{part:Clover_TripleIntersectionEmpty}.

Lemma \ref{lem:CriticalUnionsMPP} now implies $d(R,S,T)=0$ (part \ref{part:Clover_RSTCoverAllEdges}) and that $R\cup S\cup T$ is mixed critical. Which in turn implies that $R\cup S\cup T + v$ is dependent, and hence $R\cup S\cup T = V-v$. Thus proving part \ref{part:Clover_V-v}.

We now consider part \ref{part:Clover_R}. Since $G$ is a mixed circuit, it contains at least two edges of opposite type to $P$. But $S$ and $T$ only induce edges of type $P$, and $d(R,S,T)=0$, so  this implies all such edges must be induced by $R$. Hence $|R|\geq 3$, as required.

Finally, since $|S\cup T|\geq 3$, $|R|\geq 3$, $R\cap (S\cup T)=\{s,t\}$ and $d(R,S,T)=0$, we must have that $\{s,t\}$ is the 2-vertex-cut of the edge-disjoint 2-separation $(G[R],G[S\cup T+v]-E(R))$ where $G[S\cup T + v]-E(R)$ is pure. Hence proving part \ref{part:Clover_Unbalanced2Separation}.
\end{proof}

%********************
% Clover
%********************
\begin{figure}
\centering
\includegraphics{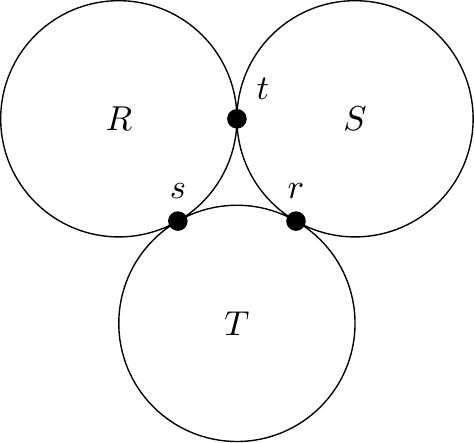}
\caption{Clover formed from vertex sets $R,S$ and $T$.}
\label{fig:Clover}
\end{figure}

Using these properties, we can determine when mixed and pure nodes of a circuit are admissible:

\begin{lem}\label{lem:NonAdmissibleMixedNode}\cite[Lemma 4.5]{JJ_MixedCircuits}
Let $G=(V;D,L)$ be a mixed circuit such that $G\not\in\{K_3^+,K_3^-\}$, and let $v$ be a mixed node of $G$. Then exactly one of the following hold:
\begin{enumerate}
  \item $v$ is admissible,
  \item $v$ has exactly two neighbours $x$ and $y$ and there exists a length critical set $R$ and a direction critical set $S$ with $R\cap S=\{x,y\}$, $R\cup S=V-v$, $d(R,S)=0$ and $i(R\cap S)=0$, or
  %NB: Know that i_G (X \cap Y) = 0, because otherwise sparsity conditions would be contradicted.
  \item there is a strong flower on $v$ in $G$.
\end{enumerate}
\end{lem}

\begin{lem}\label{lem:NonAdmissiblePureNode}\cite[Lemmas 4.4, 4.7]{JJ_MixedCircuits}
Let $G=(V;D,L)$ be a mixed circuit and let $v$ be a pure node of $G$. Then exactly one of the following hold:
\begin{enumerate}
	\item $v$ is admissible,
	\item there is a strong or weak flower on $v$ in $G$, or
	\item there is a clover on $v$ in $G$.
\end{enumerate}
\end{lem}

%PROOF THAT EXACTLY ONE OF THESE HOLD:
%(b) or (c) <=> not (a). 
%			See \cite[Lemmas 4.3, 4.7]{JJ_MixedCircuits}
%not both (b) and (c) 
%			If have both flower (R,S,T) and clover (R,S',T'). Then this implies that wlog t is a 1-separator of S'. Lemma on 1-seps then implies that S'-R is mixed which contradicts that {s,t} is an unbalanced 2-sep.

In the next section, we shall extend these results to all $\M$-connected mixed graphs by identifying when such a graph contains an admissible node in the last lobe of its ear decomposition. Lemmas \ref{lem:MixedMConnectedHasDecompIntoMixedCircuits} and \ref{lem:EarDecompRules}\ref{part:EarDecompRules_XMixedCritical} ensure that the last circuit, $C_m$, in the ear decomposition of such a graph, $G$, is mixed, and that the lobe is either a single edge, or has vertex set $V(G[C_m])-X$ for some mixed critical set $X$ in $V(G[C_m])$. So this problem reduces to describing when a mixed circuit $H$, with mixed critical set $X\subset V(H)$, has an admissible node in $V(H)-X$. In the remainder of this section, we obtain a result which will help us to identify such a node.

%DEFN - node-critical set
Let $G=(V;D,L)$ be a circuit and $v$ be a node in $G$ with $N(v)=\{x,y,z\}$. If $X$ is a critical set with $\{y,z\}\subseteq X\subseteq V-\{v,x\}$ then $X$ is called a \emph{$v$-critical set}. If, in addition, $d(x)\geq 4$ then $X$ is called \emph{$v$-node-critical}.

\begin{thm}\label{thm:LargerCriticalSetOrUnbalanced2Sep}%Extension of Lemma 4.8 in \cite{JJ_MixedCircuits}
Let $G=(V;D,L)$ be a mixed circuit and let $X$ be a mixed critical set in $V$. Suppose that either
\begin{enumerate}
  \item there is a non-admissible series node $u$ of $G$ in $V-X$ with exactly one neighbour $r$ in $X$, and $r$ is a node, or \label{part:LargerCriticalSetOrUnbalanced2Sep_SeriesNode}
  \item there is a non-admissible leaf node $u$ of $G$ in $V-X$ with $|N(u)\cap X|\leq 1$.\label{part:LargerCriticalSetOrUnbalanced2Sep_LeafNode}
\end{enumerate}
Then either there exists a mixed node-critical set $X^{*}$ with $X^{*}\supset X$, or there exists an edge-disjoint 2-separation $(H_{1},H_{2})$ of $G$ with $X\subseteq V(H_{1})$ and $H_{2}$ pure.
\end{thm}

\begin{proof}%pg.74-75
Suppose $|N(u)|=2$. Then Lemma \ref{lem:NonAdmissibleMixedNode} implies $N(u)=\{r,s\}$ is the 2-vertex-cut of an unbalanced, edge-disjoint 2-separation $(H_{1},H_{2})$ of $G$ where $H_1$ is direction critical and $H_2$ is length critical. However $X$ is mixed critical, so $G[X]$ contains both length and direction edges, which implies that $X$ intersects both $V(H_1)-V(H_2)$ and $V(H_2)-V(H_1)$. We know, by Lemma \ref{lem:CriticalSet}\ref{part:CriticalSet_2EdgeConnected}, that $G[X]$ is connected, thus $X$ must intersect $V(H_1)\cap V(H_2)= \{r,s\}$. In both cases \ref{part:LargerCriticalSetOrUnbalanced2Sep_SeriesNode} and \ref{part:LargerCriticalSetOrUnbalanced2Sep_LeafNode} we know $|X\cap N(u)|\leq 1$, so $X$ must contain exactly one neighbour of $u$, say $r$. But then $\{r\}$ is the 1-vertex-cut of the unbalanced 1-separation $(G[X\cap V(H_{1})],G[X\cap V(H_{2})])$ of $X$, which contradicts Lemma \ref{lem:CriticalSet}\ref{part:CriticalSet_1sep}.

%CASE (a): u a series node with neighbouring node r in X
Thus we must have $|N(u)|=3$. Let $N(u)=\{r,s,t\}$ and suppose condition \ref{part:LargerCriticalSetOrUnbalanced2Sep_SeriesNode} holds. Then $N(u)\cap X = \{r\}$ and both $r$ and $s$ are nodes but $t$ is not. Since $u$ is non-admissible, Lemmas \ref{lem:NonAdmissibleMixedNode} and \ref{lem:NonAdmissiblePureNode} imply there exists a critical set $T$ such that $r,s\in T$ but $t,u\not\in T$. We know that $G[V_{3}]$ is a forest by Lemma \ref{lem:NodeForest}. So since $r,s$ and $u$ are nodes, and $ru,su\in E$, this implies $rs\not\in E$. Thus Lemma \ref{lem:CriticalSet}\ref{part:CriticalSet_2EdgeConnected} implies that $G[T]$ is 2-edge-connected with $|T|\geq 3$. Hence $\delta(G[T])\geq 2$.

Since $r\in X\cap T$ is a node and $u\notin X\cup T$, in order to satisfy the minimal degree condition for $G[T]$ we must have $N(r)-\{u\}\subseteq T$. But we also know $G[X]$ is connected with $|X|\geq 2$, so some member of $N(r)-\{u\}$ must also be contained in $X$. Hence $|X\cap T|\geq 2$ and thus, by Lemma \ref{lem:CriticalUnions}, $X^*=X\cup T$ is a mixed $u$-node-critical set with $X\cup T\supset X$ since $s\in T-X$.

%CASE (b): u a leaf node and |N(u) n X|\leq 1
We now consider case \ref{part:LargerCriticalSetOrUnbalanced2Sep_LeafNode}. Since $u$ is non-admissible with $\abs{N(u)}=3$, Lemmas \ref{lem:NonAdmissibleMixedNode} and \ref{lem:NonAdmissiblePureNode} imply that there is either a strong or weak flower on $u$, or, if neither of these occur, then there is a clover on $u$.

  %SUBCASE (b) CLAIM 1:STRONG OR WEAK FLOWER	
	\begin{claim}\label{clm:LargerCriticalSetOrUnbalanced2Sep_Flower}
	If there is a strong or a weak flower on $u$ then there exists a mixed node-critical set $X^{*}$ with $X^{*}\supset X$.
	\end{claim}
  
	\begin{proof}
  	Assume there exists a strong or weak flower on $u$ with critical sets $R,S$ and $T$ such that $\{s,t\}\subseteq R\subseteq V-\{r,u\}$, $\{r,t\}\subseteq S\subseteq V-\{s,u\}$ and $\{r,s\}\subseteq T\subseteq V-\{t,u\}$. Since $u$ is a leaf node, we can assume both $r$ and $t$ are not nodes. Lemma \ref{lem:Flowers} now implies at least one of $R$ and $T$ is mixed critical so, relabelling if necessary, we can assume $T$ is mixed critical.
  	
  	Suppose $T\cap X=\emptyset$. Since $T\cup R = V-u\supset X$, we must have $R\supseteq X$, and hence $X^*=R$ is a mixed $u$-node-critical set. Further, since $R\cap N(u)=\{s,t\}$ but at most one of $s$ and $t$ is in $X$, we have $R\supset X$ as required.
  	
  	We next suppose $|T\cap X|\geq 1$ and $t\not\in X$. By Lemma \ref{lem:CriticalUnions}\ref{part:CriticalUnions_BothMixed}, $X^*=T\cup X$ is a mixed $u$-node-critical set. Additionally, since both $r,s\in T$ and at most one of these is in $X$, we have that $X^*=T\cup X \supset X$ as required.
  	
  	It remains to consider the case where $|T\cap X|\geq 1$ and $t\in X$. Since $|X\cap N(u)|\leq 1$, this implies  $r,s\not\in X$ and $|R\cap X|\geq 1$. If either $|R\cap X|\geq 2$ or $R$ is mixed critical then Lemma \ref{lem:CriticalUnions} implies that $X^*=R\cup X$ is $u$-node-critical with $X\cup R\supset X$, since $s\in R-X$. So we may assume $R$ is pure critical with $R\cap X=\{t\}$. By Lemma \ref{lem:Flowers}, we know $R\cup T=V-u$ and $d(R,T)=0$. So since $t\in R-T$, these properties imply $N(t)-\{u\}\subseteq R$. Also, since $G[X]$ is connected and $t\in X$, we know $X$ must contain some member of $N(t)-\{u\}$. Hence $|R\cap X|\geq 2$, which contradicts our assumption.
  \end{proof}

  %SUBCASE (b) CLAIM 2: UNBALANCED 2-SEPARATION
  \begin{claim}\label{clm:LargerCriticalSetOrUnbalanced2Sep_Clover}
	If there is a clover on $u$, then there exists an edge-disjoint 2-separation $(H_{1},H_{2})$ of $G$ with $X\subseteq V(H_{1})$ and $H_{2}$ pure.
	\end{claim}
	
	\begin{proof}
	Suppose there is a clover on $u$. Then $u$ is pure, and Lemma \ref{lem:Clover} implies that there exists a mixed critical set $R$ and pure critical sets $S$ and $T$ of the same type as $u$ such that $(G[R],G[S\cup T + u]-E(R))$ is an edge-disjoint 2-separation of $G$ on 2-vertex-cut $\{ s,t\}$ where $G[S\cup T + u]-E(R)$ is pure. Since $X$ is mixed critical, $X$ contains an edge $e$ of opposite type to $u$. But $S$ and $T$ only induce edges of the same type as $u$, so we must have $e\in E(R)$. Hence $|X\cap R| \geq 2$.
 
  	Assume that $X$ contains some vertex in $(S\cup T)-R$. Since $X$ is connected, and $|X\cap N(u)|\leq 1$, this implies $X$ contains exactly one of the vertices in the 2-vertex-cut $\{s,t\}$. But then this vertex will be a 1-vertex-cut of the 1-separation $(G[X\cap R], G[X\cap (S\cup T)])$ of $G[X]$, which contradicts Lemma \ref{lem:CriticalSet}\ref{part:CriticalSet_1sep}. Hence $X\subseteq R$ and $(G[R],G[S\cup T + u]-E(R))$ is the edge-disjoint 2-separation of $G$ required.
  \end{proof}
  
Claims \ref{clm:LargerCriticalSetOrUnbalanced2Sep_Flower} and \ref{clm:LargerCriticalSetOrUnbalanced2Sep_Clover}, complete our proof of case \ref{part:LargerCriticalSetOrUnbalanced2Sep_LeafNode}.
\end{proof}

We can now use Theorem \ref{thm:LargerCriticalSetOrUnbalanced2Sep} to obtain our result on mixed critical sets.

\begin{thm}\label{thm:NodeCriticalSetInMixedCircuit_AdmNodeOrUnbalanced2Sep}
Let $G=(V;D,L)$ be a mixed circuit, and $X$ be a mixed critical set in $G$. Suppose $V-X$ contains a vertex which is not a node. Then either $V-X$ contains an admissible node, or there exists an edge-disjoint 2-separation $(H_{1},H_{2})$ of $G$ with $X\subseteq V(H_{1})$ and $H_{2}$ pure.
\end{thm}

\begin{proof}
Let $X'$ be a maximal mixed critical set in $G$ such that $X'\supseteq X$ and $V-X'$ contains a vertex which is not a node. Lemma \ref{lem:MixedCircuitExclCriticalSetContainsNode} implies that $V-X'$ contains a node. Hence, by Lemma \ref{lem:NodeForest}, there exists some node $v\in V-X'$ such that $v$ is a leaf in $G[V_3-X']$. 

If $\abs{N(v)\cap X'} =3$ then $X'+v$ would break the sparsity conditions for circuits, and if $\abs{N(v)\cap X'} =2$ then $X'+v$ would be a larger $u$-node-critical set, contradicting the maximality of $X'$. So $\abs{N(v)\cap X'} \leq 1$, and either $v$ is a series node in $G$ with exactly one neighbour in $X'$, which is also a node, or $v$ is a leaf node in $G$.

If $v$ is admissible then we are done. Otherwise, Theorem \ref{thm:LargerCriticalSetOrUnbalanced2Sep} implies that either $G$ has a node-critical set $X^*$ such that $X^*\supset X'$ or $G$ has an edge-disjoint 2-separation $(H_{1},H_{2})$ with $X'\subseteq V(H_{1})$ and $H_{2}$ pure. If the former case holds, then by the definition of node-critical, $V-X^*$ contains a vertex which is not a node, which contradicts the maximality of $X'$. So the latter case must hold, as required.
\end{proof}

%#######################################################################################################
%#######################################################################################################

\section{Constructing $\M$-connected Mixed Graphs}\label{sec:Admissible}
%DEFN - admissible (updated for M-connectivity)
The aim of this section is to show that any $\M$-connected mixed graph $G\not\in\{K_3^+, K_3^-\}$, can be obtained from a smaller $\M$-connected mixed graph by an edge addition, 1-extension, or a 2-sum with a pure $K_4$. To do this, we consider an ear decomposition of $G$, and apply our results from Section \ref{sec:Circuits} to the last circuit in the ear decomposition. Berg, Jackson and Jord{\'a}n have shown the following results on the existence of admissible nodes in circuits:

\begin{lem}\label{lem:MixedCircuitHasAdmissibleNodeOr2Sum}\cite[Theorem 4.11]{JJ_MixedCircuits}
Let $G=(V;D,L)$ be a mixed circuit with $|V|\geq 4$. Then either $G$ can be expressed as a 2-sum of a mixed circuit with a pure $K_4$, or $G$ has an admissible node.
\end{lem}

\begin{lem}\cite[Theorem 3.8]{BJ_LengthCircuits}\label{lem:PureCircuitsAdmissibleNodes}
Let $G=(V;D,L)$ be a 3-connected pure circuit with $|V|\geq 5$. If $x,y\in V$ and $xy$ is an edge in $G$, then $G$ contains at least two admissible nodes in $V-\{x,y\}$.
\end{lem}

We extend this idea of admissibility to edges as well as nodes: an edge $e$ of an $\M$-connected mixed or pure graph $G$ is \emph{admissible} if $G-e$ is $\M$-connected. For $\M$-connected pure graphs we know that if a graph is 3-connected then it either contains an admissible edge or an admissible node:

\begin{lem}\label{lem:PureMConnectedAdmissibleNode}\cite[Theorem 5.4]{JJ_LengthConnected}
Let $G=(V;D,L)$ be a 3-connected, $\M$-connected pure graph. Let $C_1,C_2,\ldots , C_m$ be an ear decomposition of $\R(G)$ into pure circuits. Suppose that $G-e$ is not $\M$-connected for all $e\in\tilde{C}_m$ and for all but at most two edges of $C_m$. Then $V(G[C_m])-V(G[\bigcup_{i=1}^{m-1}C_i])$ contains an admissible node.
\end{lem}

In the remainder of this section, we obtain a similar result for admissible edges and admissible nodes in $\M$-connected mixed graphs.

%TODO - edit this - if v admissible onto xy then at least one of x,y is not contained in X (since X is mixed critical so any 1-red which adds an edge to X would not be adm.)
Consider an ear-decomposition $C_1,C_2,\ldots,C_m$ of the rigidity matroid of an $\M$-connected graph $G$. If there is some node $v$ which is admissible in $G[C_m]$, then, so long as $v$ is in the lobe of $G[C_m]$, and the edge added in the 1-reduction is not already contained in $\bigcup_{i=1}^{m-1}C_i$, the node $v$ will also be admissible in $G$. Actually, $v$ needn't be admissible in $C_m$ for this argument to work. So long as the following conditions are satisfied, $v$ will be admissible in $G$:

\begin{lem} \label{lem:MixedMConnectedAdmissibleDefinition}%KC pg.57
Let $G=(V;D,L)$ be an $\M$-connected mixed graph, and let $H_{1},\ldots,H_{m}$ be the subgraphs of $G$ induced by the ear decomposition $C_{1},\ldots,C_{m}$ of $\R(G)$ into mixed circuits, where $m\geq 2$. Let $G_{m-1}=G[\bigcup_{i=1}^{m-1}C_i]$. Let $v\in V-V(G_{m-1})$ be a node with $x,y\in N(v)$ such that we can perform a 1-reduction at $v$ onto $xy$. %NB - in other words, there is an xy-edge missing which we can add by this 1-reduction.

Let $C$ be the unique circuit in the edge set of $(H_{m})_{v}^{xy}$. If $C\cap E(G_{m-1}) \neq\emptyset$ and $E(G_{v}^{xy})-E(G_{m-1})\subset C$ then this 1-reduction is admissible in $G$.
\end{lem}

\begin{proof}%KC pg.57
Since $v\not\in V(G_{m-1})$, we know $G_{m-1}$ is a subgraph of $G_v^{xy}$. Further, $\R(G_{m-1})$ has ear decomposition $C_1,C_2,\ldots,C_{m-1}$, so is $\M$-connected by Lemma \ref{lem:EarDecomp}\ref{part:EarDecomp_Connected}. Since $C\cap E(G_{m-1})\neq\emptyset$, the definition of matroid connectivity gives that $C\cup E(G_{m-1})$ is $\M$-connected. But $C\supset E(G_v^{xy})-E(G_{m-1})$, so $E(G_v^{xy})=C\cup E(G_{m-1})$. Hence $G_v^{xy}$ is $\M$-connected, and so $v$ is admissible in $G$.
\end{proof}

In the special case where the node $v$ in our last lobe has exactly two distinct neighbours, $v$ is always admissible, so long as no edges in $G$ are admissible:

\begin{lem}\label{lem:MixedMConnectedNodeWith2NeighboursAdmissible}%KC pg. 52
Let $G=(V;D,L)$ be an $\M$-connected mixed graph, and let $H_1,\ldots,H_m$ be the subgraphs of $G$ induced by the ear decomposition $C_{1},\ldots,C_{m}$ of $\R(G)$ into mixed circuits, where $m\geq 2$. Let $Y=V(H_{m})-\bigcup_{i=1}^{m-1}V(H_{i})$ and $X=V(H_{m})-Y$. 

Suppose no edges in $G$ are admissible, and let $v\in Y$ be a node with $|N(v)|=2$. Then $v$ is admissible in $G$.
\end{lem}

\begin{proof}%KC pg. 53
Let $N(v)=\{x,y\}$, and assume $v$ is not admissible in $G$. Since $v$ is a node, $d(v)=3$, so without loss of generality, let $vx$ be a double edge and $vy$ a single edge. Since $vx$ is a double edge, $v$ must be a mixed node. 

\begin{case}
  $x,y\in X$. Lemma \ref{lem:EarDecompRules}\ref{part:EarDecompRules_EdgeCase} implies that all $xy$-edges in $G$ must also be edges in $G_{m-1}=\bigcup_{i=1}^{m-1}H_i$. Also, since $X$ is mixed critical in $H_m$, we have $i_{H_m}(X+v)=(2|X|-2)+3=2|X+v|-1$, which implies $X+v=V(H_m)$ and hence $Y=\{v\}$.
  
  Assume $G$ contains some $xy$-edge, $e$. Since $C_1,C_2,\ldots,C_{m-1}$ is an ear decomposition of $\R(G_{m-1})$, we know $G_{m-1}$ is $\M$-connected. Add the vertex $v$ to $G_{m-1}$ by a 1-extension which removes the edge $e$. The resulting graph, $G-e$, is $\M$-connected by Lemma \ref{lem:MixedMConnectedPreservedBy1extensions}. But this means $e$ is an admissible edge in $G$, which contradicts our assumption. Hence $G$ contains no $xy$-edges, and so we can perform a 1-reduction at $v$ onto the pair $vx,vy$ to form the graph $G_v^{xy}=G_{m-1}+xy$. Since $G_{m-1}$ is $\M$-connected, and edge additions preserve $\M$-connectivity by Lemma \ref{lem:MixedMConnectedPreservedByEdgeAdditions}, $G_v^{xy}$ is $\M$-connected. 
\end{case}

\begin{case}  
  $|\{x,y\}\cap X|\leq 1$. By Lemma \ref{lem:MConnectedGraphsAre2Connected}, $G$ is 2-connected, so $|X|\geq 2$ and thus $|V(H_m)|\geq 4$. Hence $H_m\not\in\{K_3^+,K_3^-\}$, and so $xy$ is not a double edge in $G$. %Otherwise G[v,x,y] would be a circuit which is a proper subgraph of the circuit H_m XX
  If $v$ is admissible in $H_m$ then we are done by Lemma \ref{lem:MixedMConnectedAdmissibleDefinition}. So assume that $v$ is not admissible in $H_m$. Then, by Lemma \ref{lem:NonAdmissibleMixedNode}, there exists a length critical set $A$ and a direction critical set $B$ with $A\cap B=\{x,y\}$, $A\cup B=V(H_m)-v$ and $d_{H_m}(A,B)=0$. By Lemma \ref{lem:EarDecompRules}\ref{part:EarDecompRules_XMixedCritical}, $X$ is mixed critical in $H_m$. Since $v\not\in X$, this implies that $X$ must intersect both $A-B$ and $B-A$. But $X$ is connected by Lemma \ref{lem:CriticalSet}\ref{part:CriticalSet_2EdgeConnected}, so this implies $|\{x,y\}\cap X|\geq 1$.
  
  Hence $|\{x,y\}\cap X|=1$, and this vertex is the 1-vertex-cut of the unbalanced 1-separation $(H_m[X\cap A],H_m[X\cap B])$ of $X$ in $H_m$, which contradicts Lemma \ref{lem:CriticalSet}\ref{part:CriticalSet_1sep}. Thus $v$ must be admissible in $H_m$, and hence also in $G$.\qedhere
\end{case}
\end{proof}

We are now in a position to prove our main result on the existence of admissible nodes and edges in $\M$-connected mixed graphs:%TODO - rephrase

\begin{thm}\label{thm:MixedMConnectedHasAdmissibleNodeEdgeOr2Sum}%KC Analogue(ish) to (3) Theorem 5.4
Let $G=(V;D,L)$ be an $\M$-connected mixed graph such that $G\not\in\{K_3^+,K_3^-\}$ and $G$ has no admissible edges. Let $H_1,H_2,\ldots,H_m$ be the subgraphs of $G$ induced by an ear decomposition $C_{1},C_{2},\ldots, C_{m}$ of $\R(G)$ into mixed circuits. Let $E_j=\bigcup_{i=1}^{j}C_i$, $Y=V(H_{m})-\bigcup_{i=1}^{m-1}V(H_{i})$ and $X=V(H_{m})-Y$.
Then either $Y$ contains an admissible node, or $G$ can be expressed as the 2-sum of a mixed $\M$-connected graph with a pure $K_4$.
\end{thm}

\begin{proof}
%EDITS:
%pre-20140916 1042 	- Thm on D-bal M-conn mixed graphs (admissible edge, node or 2-sum with D-pure K_4)
%		 20140916 1042	- extended to all M-conn mixed graphs (adm. edge, node or 2-sum with D/L-pure K_4)
If $G$ is a mixed circuit, or $G=F_{1}\oplus_2 F_2$ where $F_2$ is a pure $K_4$, then we are done by Lemmas \ref{lem:MixedCircuitHasAdmissibleNodeOr2Sum} and \ref{lem:2SumMConn}\ref{part:2SumMConn_Cleave} respectively. So assume that $G$ is not a circuit and cannot be expressed as such a 2-sum. Since $G$ has no admissible edges, Lemma \ref{lem:EarDecompRules} implies $Y\neq\emptyset$, and hence $X$ is mixed critical in $H_m$. Lemma \ref{lem:MixedCircuitExclCriticalSetContainsNode} now implies that $Y$ contains a node. Since, by Lemma \ref{lem:NodeForest}, the node-subgraph of a mixed circuit is a forest, we can find a node $u\in Y$ such that $u$ is a leaf in $G[V_{3}\cap Y]$. 

If $|N(u)|=2$, then Lemma \ref{lem:MixedMConnectedNodeWith2NeighboursAdmissible} implies that $u$ is admissible in $G$ and we are done. So assume that $|N(u)|=3$ and let $N(u)=\{r,s,t\}$.

	%CASE 1: N(v)\subseteq X
\begin{case}
	$N(u)\subseteq X$. Since $X$ is mixed critical in $H_m$, the sparsity conditions imply $X+u=V(H_m)$, and hence $Y=\{u\}$. In order to perform a 1-reduction on $u$ in $G$, we first have to make sure that there is a pair of vertices in $N(u)$ that we can 1-reduce onto i.e.\ if $u$ is pure of type $P\in \{D,L\}$, then there must be a pair of vertices in $N(u)$ which are not connected by an edge of type $P$; and if $u$ is mixed, then there must exist a pair of vertices in $N(u)$ which have at most one edge between them.
	%NB: we know that such a pair must exist in H_m[N(u)], but it is possible that there are extra edges between these vertices in G (since G is M-connected and 1-extensions preserve M-connectivity, it is possible that all the edges in N(u) could have been added).
	
	Suppose $u$ is a pure node of type $P$, and assume $G[N(u)]$ contains two edges, $e$ and $f$, of the same type as $u$. Since $e$ and $f$ are not parallel, they must cover all three vertices in $N(u)$. We know $G[E_{m-1}]$ is $\M$-connected, so for all edges $g\in E_{m-1}-e$ there is some circuit $C_g\subseteq E_{m-1}$ such that $e,g\in C_g$. If $f\not\in C_g$ then $C_g$ is a circuit in $\R(G-f)$. Otherwise $f\in C_g$ and, by Lemma \ref{lem:PureMConnectedPreservedBy1extensions} when $C_g$ is pure, or Lemma \ref{lem:MixedCircuitPreservedBy1extensions} when $C_g$ is mixed, the 1-extension, $(C_g-f)\cup\{ur,us,ut\}$ is a circuit in $\R(G-f)$. So for all edges $g$ in $G-f$, we can find a circuit in $G-f$ containing both $g$ and $e$. Thus, by the transitivity of matroid connectivity, $G-f$ is $\M$-connected. This contradicts the hypothesis that $G$ contains no admissible edges, so our assumption must be false, and $G$ must contain at most one edge of type $P$ in $G[N(u)]$.
	
	Suppose instead that $u$ is a mixed node, and assume that each pair of vertices in $N(u)$ is connected by a double-edge. Then $G[N(u)]$ contains two non-parallel edges $e$ and $f$ such that $e\in D$ and $f\in L$. By a similar argument to the pure case above, we can show that $G-f$ is $\M$-connected, once more contradicting our hypothesis, and thus showing that $N(u)$ contains a pair of vertices with at most one edge between them.
	
	So regardless of whether $u$ is mixed or pure, we can find some pair of vertices, say $\{r,s\}$, in $N(u)$ such that we can perform a 1-reduction at $u$ onto $rs$. The resulting graph, $\bigcup_{i=1}^{m-1} H_i+rs$ is $\M$-connected, by Lemmas \ref{lem:EarDecomp}\ref{part:EarDecomp_Connected} and \ref{lem:MixedMConnectedPreservedByEdgeAdditions}. Hence $u$ is admissible in $G$.
	\end{case}
	
	% CASE 2: |N(u)\cap X| = 2
\begin{case}
$|N(u)\cap X|=2$. Let $N(u)\cap X = \{r,s\}$. Since $X$ is mixed critical in $H_m$, we have no admissible 1-reduction at $u$ onto $rs$ in $H_m$. If $u$ has an admissible 1-reduction in $H_m$ onto either $rt$ or $st$, then $u$ is also admissible in $G$ by Lemma \ref{lem:MixedMConnectedAdmissibleDefinition}, and we are done. So suppose $u$ is not admissible in $H_m$. Then Lemmas \ref{lem:NonAdmissibleMixedNode} and \ref{lem:NonAdmissiblePureNode} imply that there exists a triple $(R,S,T)$ which is either a flower or a clover on $u$ in $H_m$. We may assume that $R, S$ and $T$ are minimal sets with this property. We also know that $X$ is a mixed critical set with $r,s\in X\subseteq V(H_m)-\{t,u\}$, so by the definitions, the triple $(X,R,S)$ is either a strong or a weak flower, or a clover on $u$ in $H_m$.

	Suppose $(X,R,S)$ is either a strong or a weak flower on $u$ in $H_m$. Then Lemma \ref{lem:Flowers} implies $d_G(S,X)=0$, so $G$ contains no $st$-edges. Hence we can perform a 1-reduction at $u$ onto $st$. Let $C$ denote the unique circuit in $E((H_m)^{st}_u)$ formed by this 1-reduction. Since $R$ is a minimal critical set in $H_m$ which contains both $s$ and $t$, $G[C]$ must have $R$ as its vertex set. By Lemma \ref{lem:Flowers}, $R\cup X = V(H_m)-u$ and $d_{H_m}(R,X)=0$. This implies that $C\supseteq E((H_m)^{st}_u)-E_{m-1}$. Also, since $X\cap R\cap S\neq \emptyset$ and $s\in(X\cap R)-S$, we have that $|R\cap X|\geq 2$ which, by Lemma \ref{lem:CriticalUnions}, implies that $R\cap X$ is critical in $H_m$, and hence has non-empty edge set in $H_m$. Thus  $C\cap E_{m-1}\neq\emptyset$ and so Lemma \ref{lem:MixedMConnectedAdmissibleDefinition} implies that $u$ is admissible in $G$.
	
	Suppose instead that $(X,R,S)$ is a clover on $u$ in $H_m$. Then Lemma \ref{lem:Clover} implies that $H_m$ has a 2-separation $(G[X],G[R\cup S +u]-E(X))$ where $G[R\cup S +u]-E(X)=G[\tilde{C}_m]$ is pure. Which implies that $(G[E_{m-1}],G[\tilde{C}_m])$ is an unbalanced 2-separation of $G$. 

	Let $(S_1,S_2)$ be a 2-separation of $G$ with $S_2$ minimal such that $V(S_2)\subseteq R\cup S+u$, and let the corresponding 2-vertex-cut be $\{x,y\}$. Since $G[\tilde{C}_m]$ is pure, $S_2$ must also be pure. If $G$ contains an $xy$-edge, $e$, of the same type as $S_2$, then $G-e$ is also an $\M$-connected mixed graph by Lemma \ref{lem:2SumMConn}\ref{part:2SumMConn_Cleave}, which contradicts our assumption that $G$ contains no admissible edges.	Thus $e\not\in E(G)$, and so, by Lemma \ref{lem:2SumMConn}, $G=F_1\oplus_2 F_2$ where $F_1$ and $F_2$ are $\M$-connected and are formed from $S_1$ and $S_2$ respectively by adding the $xy$-edge $e$. Since $S_2$ is minimal and $F_2 \neq K_4$, we know that $F_2$ must be a 3-connected, $\M$-connected pure graph on at least 5 vertices. Further, since $S_2\subseteq G[\tilde{C}_m]$, Lemma \ref{lem:2SumMixedCircuit} implies that $F_2$ is a pure circuit. %NB: know F_2 must be a circuit since \subset Y\cup{r,s}
	
		Hence, by Lemma \ref{lem:PureCircuitsAdmissibleNodes}, $F_2$ has at least two admissible nodes in $V(F_2)-\{x,y\}$. Let $F_2'$ be the graph formed by an admissible 1-reduction at one of these nodes. Then $G'=F_1 \oplus_2 F_2'$ is an $\M$-connected mixed graph, and is a 1-reduction of $G$. Hence this node is also admissible in $G$. 
\end{case}

\begin{case}
	% CASE 3: |N(u)\cap X| =< 1
	$|N(u)\cap X|\leq 1$. If $u$ is admissible in $H_m$, then we are done by Lemma \ref{lem:MixedMConnectedAdmissibleDefinition}. So suppose $u$ is not admissible in $H_m$. Since $u$ is a leaf in $G[V_3\cap Y]$, $u$ has some neighbour in $Y$ which is not a node. Thus, by Theorem \ref{thm:NodeCriticalSetInMixedCircuit_AdmNodeOrUnbalanced2Sep}, either $Y$ contains an admissible node $v$ and we are done; or there is a 2-separation $(S_{1},S_{2})$ of $H_{m}$ such that $S_{2}$ is pure and $X\subseteq V(S_{1})$, in which case $(S_{1}\cup G[E_{m-1}],S_{2})$ is an unbalanced 2-separation of $G$. By the same argument used in the case where $|N(u)\cap X|=2$ and $G$ had an unbalanced 2-separation, we can find a 2-separation $(S_1',S_2')$ of $G$ on some 2-vertex-cut $\{x,y\}$, where $S_2'$ is minimal with $S_2'\subseteq S_2$ and there is a node in $V(S_2')-\{x,y\}$ which is admissible in $G$.
\qedhere 
\end{case}
\end{proof}

We know, by Lemmas \ref{lem:MixedMConnectedPreservedByEdgeAdditions}, \ref{lem:MixedMConnectedPreservedBy1extensions} and \ref{lem:2SumMConn}\ref{part:2SumMConn_Sum}, that the inverse of the operations used in Theorem \ref{thm:MixedMConnectedHasAdmissibleNodeEdgeOr2Sum} preserve $\M$-connectivity. Thus the following inductive construction immediately follows from this result:

\begin{thm}
Let $G$ be a mixed graph. Then $G$ is $\M$-connected if and only if $G$ can be obtained from $K_3^+$ or $K_3^-$ by a sequence of edge additions, 1-extensions and 2-sums with pure $K_4$'s.
\end{thm}

%#######################################################################################################
%#######################################################################################################

\section{Constructing Direction-Balanced $\M$-connected Mixed Graphs} \label{sec:Feasible}

In this section, we show that any direction-balanced, $\M$-connected mixed graph can be constructed from either $K_3^+$ or $K_3^-$ by a sequence of edge additions, 1-extensions and 2-sums with direction-pure $K_4$'s. This construction will be used to characterise global rigidity for $\M$-connected graphs in Section \ref{sec:GlobalRigidity}.

%DEFN - feasible
To obtain this construction, we shall show that if a direction-balanced, $\M$-connected mixed graph $G$ cannot be obtained by a 2-sum with a direction-pure $K_4$, then $G$ either has an admissible node or an admissible edge, whose removal preserves being direction-balanced. As such, we define a vertex $v$ of $G$ to be \emph{feasible} if there is an admissible 1-reduction at $v$ which preserves being direction balanced, and define and edge $e$ to be \emph{feasible}, if it is admissible and $G-e$ is direction-balanced.

Before proving this result, we first make the following observations about graphs with unbalanced 2-separations:

\begin{lem}\label{lem:2VertexCutPreservedByEdgeDeletionAnd1Red}
Let $G=(V;D,L)$ be an $\M$-connected mixed graph with an edge-disjoint 2-separation $(H_1,H_2)$ on 2-vertex-cut $\{ x,y\}$ with $H_2$ pure. Let $G'$ be formed from $G$ by either an edge-deletion or a pure 1-reduction. Then $\{x,y\}$ is also a 2-vertex-cut in $G'$. %Pure 1-reduction because need to ensure |N(v)|=3. Could just add this neighbourhood condition instead.
\end{lem}
%NB: G' needn't be M-connected, but G must be to ensure d_G(x),d_G(y)\geq 4

\begin{proof}
This trivially holds when $G'$ is formed from $G$ by an edge deletion, so we shall only prove the case where $G$ is formed by a 1-reduction. Suppose $G'=G_v^{rs}$ for some node $v\in V$. By Lemma  \ref{lem:2SumMConn}\ref{part:2SumMConn_Cleave}, $d_G(x),d_G(y)\geq 4$, which implies $v\not\in\{x,y\}$. So without loss of generality, we can assume $v\in V(H_2)-V(H_1)$. Since $\{x,y\}$ is a 2-vertex-cut in $G$, this implies $N_G(v)\subset V(H_2)$. Hence $(H_1,(H_2)_v^{rs})$ is a 2-separation of $G'$ on the 2-vertex-cut $\{x,y\}$.
\end{proof}

\begin{lem}\label{lem:UnbalancedLengthRedToDirectionRed}%equiv to \citep{JJ_MixedCircuits} Lemma 4.13 for mixed circuits
Let $G=(V;D,L)$ be a direction-balanced, $\M$-connected mixed graph and let $v$ be a mixed node of $G$ with $N(v)=\{r,s,t\}$, where potentially $t=r$. Suppose $G'$ is the graph formed by an admissible length 1-reduction at $v$ onto the edge $l=rs$, and that $G'$ is not direction-balanced i.e.\ $G'$ has a 2-separation $(H_1,H_2)$ with $H_2$ length-pure. Suppose $l\in E(H_2)$. Then either %the direction edge added has both endvertices in $V(H_2)$ and 
\begin{enumerate}
	\item $\{r,s\}=V(H_1)\cap V(H_2)$. In which case $t\in V(H_2)-V(H_1)$, and $G$ has admissible direction 1-reductions onto both $rt$ and $st$; or \label{part:UnbalancedLengthRedToDirectionRed_Onto2sep}
	\item $\{r,s\}$ intersects $V(H_2)-V(H_1)$, and $G$ has an admissible direction 1-reduction onto $rs$. \label{part:UnbalancedLengthRedToDirectionRed_NotOnto2sep}
\end{enumerate}
%TODO WIP - keep or remove the following?
%%In addition, if the graph $G''$ formed by one of these admissible direction 1-reductions is not direction-balanced, then for all length-pure ends $X$ in $G''$, $X$ is also a length-pure end of $G'$, and either $X\subset V(H_1)-V(H_2)$ or  $X\subset V(H_2)-V(H_1)$. 
\end{lem}

\begin{proof}%pg. 87
Let $V_1$ and $V_2$ denote the vertex sets of $H_1$ and $H_2$ respectively, and let $\{x,y\}=V_1\cap V_2$. Since $G$ is direction-balanced, $N_{G}(v)\cap(V_2-V_1)\neq\emptyset$.

We first consider case \ref{part:UnbalancedLengthRedToDirectionRed_Onto2sep}, where $\{r,s\}=\{x,y\}$. Lemma \ref{lem:2SumMConn}\ref{part:2SumMConn_Cleave} implies $G'-l$ is $\M$-connected. Since $N_{G}(v)\cap(V_2-V_1)\neq\emptyset$, we must have $t\in V_2-V_1$. But $H_2$ is length-pure, so this implies neither $st$ nor $rt$ are direction edges in $G'$ or $G$. Hence, since edge additions preserve $\M$-connectivity by Lemma \ref{lem:MixedMConnectedPreservedByEdgeAdditions}, we can add either of the direction edges $st$ or $rt$ to $G'-l$ to obtain $G^*$, which is a direction 1-reduction at $v$ as required.

We now consider case \ref{part:UnbalancedLengthRedToDirectionRed_NotOnto2sep}, where $\{r,s\}\neq\{x,y\}$. Since $l\in E(H_2)$, this implies at least one of the endvertices of $l$, say $s$, is contained in $V_2-V_1$. But $H_2$ is length-pure, so $rs$ is not a direction edge in $G'$ or $G$. Let $G^*$ be the graph obtained from $G$ by a direction 1-reduction onto the edge $rs$, and call this direction edge $d$. It remains to show that $G^*$ is $\M$-connected.

By construction, we have $E(G^*)=E(G')-l+d$. Let $f$ be a direction edge in $E(G^*)\cap E(G')$. Since $G'$ is $\M$-connected, we know that for all $e\in E(G')-f$, there is a circuit $C'$ in $\R(G')$ such that $e,f\in C'$. If $l\not\in C'$, then $C'$ is also a circuit in $\R(G^*)$ and we are done. Otherwise $l\in C'$, and so $C'$ is mixed. In which case, Lemma \ref{lem:MixedCircuitCanChangePureVertexToMixedVertex} implies that the edge set $C^*=C-l+d$ is a mixed circuit in $\R(G^*)$ containing $f$ and $d$ (and $e$, when $e\neq l$). Thus, by the transitivity of $\M$-connectivity, $G^*$ is $\M$-connected. Hence the direction 1-reduction at $v$ onto $rs$ is admissible.
%NB: In this case, G'' could have an unbalanced 2-separation where H_2'\subseteq H_1. If G' looks like (L):(M):(L)
%TODO - WIP - need to add proof that L-pure ends of G'' are also L-pure ends in G' satisfying containment criteria.
\end{proof}

We are finally in a position to prove our main result:

\begin{thm}\label{thm:DbalMConnHasFeasibleMove}
Let $G\not\in\{K_3^+, K_3^- \}$ be a direction-balanced, $\M$-connected mixed graph. Suppose $G$ cannot be expressed as the 2-sum of a direction-balanced mixed graph with a direction-pure $K_4$. Then $G$ has either a feasible edge or a feasible vertex.
\end{thm}

\begin{proof}

We proceed by contradiction, i.e.\ by assuming there exists some mixed graph $G$ such that all 1-reductions and edge-deletions of $G$ which preserve $\M$-connectivity do not preserve being direction-balanced. Theorem \ref{thm:MixedMConnectedHasAdmissibleNodeEdgeOr2Sum} tells us that we are able to construct some smaller $\M$-connected graph $G'$ from $G$, either by deleting an admissible edge, or by performing an admissible 1-reduction at some node in $G$. Let
\[
n(G')=\abs{ \{ v\in V(G'): v\in X \text{ for some length-pure end } X \text{ of } G'\} }.  
\]
%TODO - WIP. Do we need to worry about all ends? Or can we just use the argument that we can always find a move that will make this end smaller? Pick $V(H_2)$ minimal wrt inclusion
We assume that out of all choices of admissible edges and nodes in $G$, the admissible move which formed $G'$ is such that $n(G')$ is minimal. Since we are also assuming that $G$ has no feasible moves, $G'$ is not direction-balanced, and hence $n(G')>0$.

Note that if $G'$ was formed by a length 1-reduction at $v$ onto the edge $rs$, and $G$ also has an admissible direction 1-reduction at $v$ onto $rs$, then every length-pure end in the graph $G_d$ formed by the direction 1-reduction is also a length-pure end in $G'$. Hence $n(G_d)\leq n(G')$, and so we can assume $G'$ was chosen to be the direction 1-reduction instead. 

Let $(H_1, H_2)$ be a 2-separation of $G'$  on some 2-vertex-cut $\{x,y\}$, such that $H_2$ is length-pure, and is minimal with respect to inclusion.

	\begin{claim}\label{clm:FeasibilityProof_xyNotLengthEdge}
	$xy$ is not a length edge in $G'$.
	\end{claim}
	
	\begin{proof}
	Assume $xy$ is a length edge in $G'$. Then either $xy$ is also a length edge in $G$, or this edge was added in a length 1-reduction of $G$ and so $G'=G_v^{xy}$ for some $v\in V(G)$.
	
	First, suppose $xy$ is a length edge of $G$. Then Lemma \ref{lem:2SumMConn}\ref{part:2SumMConn_Cleave} implies $G'-xy$ is $\M$-connected. We can construct $G-xy$ from $G'-xy$ by performing the inverse operation to that which formed $G'$ (an edge addition when $G'=G-e$ for some $e\in E(G)$, or a 1-extension when $G'=G_v^{rs}$ for some $v\in V(G)$). By Lemmas \ref{lem:MixedMConnectedPreservedByEdgeAdditions} and \ref{lem:MixedMConnectedPreservedBy1extensions} respectively, both of these operations preserve $\M$-connectivity, so $G-xy$ is $\M$-connected. Hence $xy$ is admissible in $G$. But we assumed $G$ had no feasible edges, so $G-xy$ must have a 2-separation $(H_1',H_2')$ on some 2-vertex-cut $\{x',y'\}$ where $H_2'$ is length-pure and $\{x',y'\}$ separates $x$ and $y$. Lemma \ref{lem:2VertexCutPreservedByEdgeDeletionAnd1Red} implies $\{x',y'\}$ is also a 2-vertex-cut in $G'-xy$. Thus $G'-xy$ has two crossing 2-separators, $\{x,y\}$ and $\{x',y'\}$, which contradicts Lemma \ref{lem:UnbalancedMixedMConnHasNoCrossing2Sep}.
	
	We now consider the second case: that $xy$ was added in a 1-reduction on $v$. If $v$ is a pure node, then in order for $G$ to be direction-balanced, $v$ must have neighbours in both $V(H_1)-V(H_2)$ and $V(H_2)-V(H_1)$, in addition to neighbouring $x$ and $y$. But this contradicts that $d_G(v)=3$. Hence $v$ must be a mixed node. Let $N_G(v)=\{x,y,z\}$. Lemma \ref{lem:UnbalancedLengthRedToDirectionRed} implies $z\in V(H_2)-V(H_1)$ and that $G$ has a different admissible 1-reduction at $v$, which adds the direction edge $xz$ instead, to form the graph $G^*$.
	
By our original assumption, $G^*$ is not direction-balanced, so it has a 2-separation $(H_1^*,H_2^*)$ on some 2-vertex-cut $\{x^*,y^*\}$ with $H_2^*$ length-pure and minimal with respect to inclusion. However, by the construction above, $G^*=G'-xy+xz$, and so $\{x,y\}$ is also a 2-vertex-cut in $G^*$. Since $xz$ is a direction edge, the 2-separation given by $\{x,y\}$ in $G^*$ is direction-balanced. Hence $\{x,y\}\neq \{x^*,y^*\}$. Lemma \ref{lem:UnbalancedMixedMConnHasNoCrossing2Sep} now implies that these two 2-separations of $G^*$ cannot cross. So either $V(H_2^*)\subset V(H_1)$ and  $V(H_2^*)-\{x^*,y^*\}$ is a length-pure end of both $G'$ and $G^*$, or $V(H_2^*)-\{x^*,y^*\}\subset V(H_2)-\{x,y,z\}$. %In both these cases, the containment is proper as $\{x,y\}\neq \{x^*,y^*\}$.
Hence $n(G^*)< n(G')$, which contradicts our choice of $G'$.
	\end{proof}

Claim \ref{clm:FeasibilityProof_xyNotLengthEdge} and Lemma \ref{lem:2SumMConn}\ref{part:2SumMConn_Cleave} imply $G'=F_1\oplus_2 F_2$ where $F_2$ is length-pure, and both $F_1$ and $F_2$ are $\M$-connected and are formed from $H_1$ and $H_2$ respectively by adding the length edge $xy$. Let
\[V^*= \begin{cases}
				V(F_2)-\{x,y,r,s\}		& \text{if } G'=G-e \text{ and } e=rs, \\
				V(F_2)-\{x,y,r,s,t\}	& \text{if } G'=G_v^{rs} \text{ and } N_G(v)=\{r,s,t\} ,
			\end{cases}	\]
and
\[E^*= \begin{cases}
  			E(F_2)-\{xy\} 				& \text{if } G'=G-e \text{ and } e=rs, \\
  			E(F_2)-\{xy, rs\} 		& \text{if } G'=G_v^{rs} \text{ and } N_G(v)=\{ r,s,t\}.
		\end{cases}	\]

	\begin{claim}\label{clm:FeasibilityProof_NoAdmissibleMovesInE*V*}

	$F_2$ has no admissible edges $f\in E^*$, and no admissible nodes $w\in V^*$.
	\end{claim}
	
	\begin{proof}%pg. 90-92
	We shall show that if such an admissible edge or node exists in $F_2$, then it is also admissible in $G$ and contradicts our choice of $G'$.
	
	%PART A - f or w is admissible in G
	Assume that $F_2$ has either an admissible edge $f\in E^*$, or an admissible node $w\in V^*$ with neighbourhood $\{m,n,p\}$ such that $w$ has an admissible 1-reduction onto the edge $mn$. Then $F_2-f$, respectively $(F_2)_w^{mn}$, is $\M$-connected and contains the length edge $xy$. Thus by Lemma \ref{lem:2SumMConn}\ref{part:2SumMConn_Sum}, we can 2-sum this graph with $F_1$ to obtain
	\[(G')^* = \begin{cases}
							G'-f = F_1 \oplus_2 (F_2-f)			&\text{when } f \text{ is admissible in } F_2 \text{, or}\\
							(G')_w^{mn}=F_1 \oplus_2 (F_2)_w^{mn}	&\text{when } w \text{ is admissible in } F_2.		
					 \end{cases} \]	
	Since 2-sums preserve $\M$-connectivity, the resulting graph $(G')^* $ is $\M$-connected. By the definitions of $E^*$ and $V^*$, we know $(G')^*$ either contains both $r$ and $s$ when $G'=G-e$, or contains the vertices $r,s,t$ and the edge $rs$ when $G'=G_v^{rs}$.	Further, when $G'=G-e$ then we cannot have $e\in E((G')^*)$, as this would imply that $e$ was added back to $G'$ by a length-pure 1-reduction at $w$; for this to be possible, both endvertices of $e$ must be in $H_2$, which, since $G$ was direction-balanced, implies that $e$ is a direction edge, and thus cannot be added by such a 1-reduction.
	
	%but both endvertices of $e$ are in $H_2$, so for $G$ to be direction-balanced, $e$ must be a direction edge, and thus cannot be added in a length-pure 1-reduction.  

	Hence we can either add back the edge $e$ to $(G')^*$, or perform a 1-extension on $(G')^*$ to add back the vertex $v$, as relevant. Both of these operations preserve $\M$-connectivity by Lemmas \ref{lem:MixedMConnectedPreservedByEdgeAdditions} and \ref{lem:MixedMConnectedPreservedBy1extensions} respectively. So the resulting graph $G^*$ (where $G^*=G-f$ or $G^*=G_w^{mn}$ respectively) is $\M$-connected. Thus implying that $f$ (respectively $w$) is admissible in $G$.

By our original assumption, $G$ has no feasible nodes or edges. So $G^*$ must have a 2-separation $(H_1^*,H_2^*)$ on some 2-vertex-cut $\{x^*,y^*\}$ where $H_2^*$ is length-pure, and is minimal with respect to inclusion. Since $F_2$ is length-pure, $f\in E^*$ is a length edge (resp.\ $w\in V^*$ is a length-pure node). But $G$ is direction-balanced, so this implies that when $G^*=G-f$, the set $\{x^*,y^*\}$ separates the endvertices of $f$ in $G^*$, and when $G^*=G_w^{mn}$, the set $\{x^*,y^*\}$ separates $N_G(w)$ in $G^*$. Since $f\in E(H_2)$ in the former case, and $\{w\}\cup N_G(w)\subseteq V(H_2)$ in the latter, this implies $\{x,y\}\neq \{x^*, y^*\}$ and that $V(H_2)-V(H_2^*)\neq\emptyset$. Lemmas \ref{lem:2VertexCutPreservedByEdgeDeletionAnd1Red} and \ref{lem:UnbalancedMixedMConnHasNoCrossing2Sep} now imply that both $\{x,y\}$ and $\{x^*,y^*\}$ are 2-vertex-cuts in $(G')^*$ and do not cross. 

When $G'=G-e$, we clearly have $V(H_2^*)\subset V(H_2)$. When instead, we have $G'=G_v^{rs}$, the fact that $G$ is direction-balanced implies that $v$ is either a mixed node in $G$, or $N_G(v)$ intersects $V(H_1)-\{x,y\}$; both of which imply that $v\not\in V(H_2^*)$ and thus  $V(H_2^*)\subset V(H_2)$. Hence for all choices of $G'$ and $G^*$, $V(H_2^*)\subset V(H_2)$. Since $H_2$ and $H_2^*$ were chosen to be minimal with respect to inclusion in $G'$ and $G^*$ respectively, the definition of an end now implies that there exists some vertex in $V(H_2)-V(H_2^*)$ which is contained in some end of $G'$, but in no end of $G^*$. %TODO - may require more explanation
Hence $n(G^*) < n(G')$, which contradicts our choice of $G'$. %NB: because ends cannot overlap/peoperly contain each other.
	\end{proof}

Claim \ref{clm:FeasibilityProof_NoAdmissibleMovesInE*V*} tells us that all admissible edges or nodes in $F_2$ must be contained in $E(F_2)-E^*$ or $V(F_2)-V^*$ respectively. In the remainder of the proof, we show that whatever the structure of $F_2$, we can find such an admissible move, and this move will contradict our choice of $G'$.

	\begin{claim} \label{clm:FeasibilityProof_IfNotK4ThenFeasible}
	$F_2=K_4$.
	\end{claim}

	\begin{proof}
	% OUTLINE:
	% G-e:
	%		A: e=rs		\in H_2									(Show if endvertex of e is adm. there is a way of adding vertex back 																					 along with e so diff feasible edge deletion)
	%		B: r			\in H_2									(Show we can find admissible node in V(H_2)-{u,v,r})
	% G_v^xy:
	%		C: x,y,z 	\in H_2 	(xy \in L)		(Show:in this case v is mixed and there is a 1-red at v with xy\in D)
	%		D: x,y 		\in H_2		(xy \in L)		(Show: only possible when v is length-pure and then can add back v to 																				give other adm 1-red of G)
	%		E: z			\in H_2									(Show we can find admissible node in V(H_2)-{u,v,z})
	Assume $F_2\neq K_4$. We show that we can either find a feasible move in $G$, thus contradicting our original assumption that $G$ has no feasible moves; or we can find a different admissible move in $G$ which contradicts our choice of $G'$. As before, there are two cases to consider: when $G' = G-e$, and when $G'= G^{rs}_v$.
	
	\begin{case}
		$G'=G-e$. Since $G$ is direction-balanced, at least one of the endvertices of $e$, say $r$, is contained in $V(H_2)-V(H_1)$, and the other endvertex, $s$, is either also contained in $V(H_2)$, in which case $e$ must be a direction edge; or in $V(H_1)-V(H_2)$ and $e$ can be either a length or a direction edge.
		
		We know $F_2$ is an $\M$-connected length-pure graph. If $F_2$ is a length-pure circuit, then Lemma \ref{lem:PureCircuitsAdmissibleNodes} implies $F_2$ contains an admissible node in $V(F_2)-\{x,y,r\}$. Otherwise, by Lemma \ref{lem:EarDecomp}\ref{part:EarDecomp_Extending}, we can build an ear decomposition of $F_2$ such that the first circuit in the ear decomposition contains both the edge $xy$ and some edge incident with $r$. Lemma \ref{lem:PureMConnectedAdmissibleNode} now implies that $F_2$ contains either an edge other than $xy$ which is admissible, or an admissible node in $V(F_2)-\{x,y,r\}$. Thus, in both cases, $F_2$ either contains an admissible edge $f\in E(F_2)-\{xy\}=E^*$, or an admissible node $w\in V(F_2)-\{x,y,r\}$. Since $V^*=V(F_2)-\{x,y,r,s\}$, and, by Claim \ref{clm:FeasibilityProof_NoAdmissibleMovesInE*V*}, $V^*$ and $E^*$ contain no admissible nodes or edges respectively, it only remains to consider the case where $s\in V(H_2)-V(H_1)$ is an admissible node.
		
		Let $N(s)=\{m,n,p\}$ in $F_2$, and suppose, without loss of generality, that $(F_2)_s^{mn}$ is $\M$-connected. We can now add the vertex $s$ back to $(F_2)_s^{mn}$ by performing a 1-extension which deletes the edge $mn$ and adds back the edges $ms, ns$ and $rs=e$. The resulting graph, $F_2-ps$, is $\M$-connected by Lemma \ref{lem:PureMConnectedPreservedBy1extensions}. Hence $ps\in E^*$ is an admissible edge in $F_2$, contradicting Claim \ref{clm:FeasibilityProof_NoAdmissibleMovesInE*V*}.
	\end{case}

	\begin{case}
		$G'=G_{v}^{rs}$. First, suppose the edge, $rs$, added in the 1-reduction, is contained in $E(H_1)$. By Claim \ref{clm:FeasibilityProof_xyNotLengthEdge}, $\{r,s\}\neq\{x,y\}$, so without loss of generality, $r\in V(H_1)-V(H_2)$. In order for $G$ to be direction-balanced, we must have $t\in V(H_2)-V(H_1)$. If $F_2$ is a circuit then, by Lemma \ref{lem:PureCircuitsAdmissibleNodes}, it contains an admissible vertex in $V(H_2)-\{x,y,t\}=V^*$, contradicting Claim \ref{clm:FeasibilityProof_NoAdmissibleMovesInE*V*}. Otherwise, we can build an ear decomposition of $F_2$ whose first ear contains both $xy$ and an edge incident with $t$. Lemma \ref{lem:PureMConnectedAdmissibleNode} then implies that there is either an admissible edge in $E(F_2)-\{xy\}=E^*$, or an admissible node in $V(H_2)-\{x,y,t\}=V^*$, once more contradicting Claim \ref{clm:FeasibilityProof_NoAdmissibleMovesInE*V*}.
		
		Hence $rs\in E(H_2)-E(H_1)$ and, without loss of generality, $r\in V(H_2)-V(H_1)$. Since $H_2$ is length-pure, $rs$ is a length edge. If $v$ is a mixed node in $G$, then Lemma \ref{lem:UnbalancedLengthRedToDirectionRed} implies that $G$ has a direction 1-reduction onto $rs$. But this contradicts our choice of $G'$: that a length 1-reduction at $v$ was chosen only when $G$ had no direction 1-reduction at $v$ onto the same pair of vertices. Hence $v$ must be a length-pure node, and thus, since $r,s\in V(H_2)$, we must have $t\in V(H_1)-V(H_2)$ in order for $G$ to be direction-balanced. 
		
		Suppose $F_2$ is not a circuit. Then we can build an ear decomposition of $F_2$ such that the first ear contains both the edges $xy$ and $rs$. Lemma \ref{lem:PureMConnectedAdmissibleNode} now implies either $E(F_2)-\{xy,rs\}=E^*$ contains an admissible edge, or $V(F_2)-\{x,y,r,s\}=V^*$ contains an admissible node, contradicting Claim \ref{clm:FeasibilityProof_NoAdmissibleMovesInE*V*}.	Thus $F_2$ must be a length-pure circuit. In which case, Lemma \ref{lem:PureCircuitsAdmissibleNodes} implies that $F_2$ contains an admissible node in $V(F_2)-\{x,y,r\}\subseteq V^*\cup \{s\}$. By Claim \ref{clm:FeasibilityProof_NoAdmissibleMovesInE*V*}, the only possibility is that $s\in V(F_2)-\{x,y\}$ is admissible in $F_2$. We know $e=rs\in E(F_2)$. Let $m$ and $n$ denote the other two neighbours of $s$ in $F_2$.
		
		Since $s$ is admissible, either $(F_2)_s^{mn}$ or $(F_2)_s^{mr}$ is $\M$-connected (relabelling $m$ and $n$ if necessary). What's more, since $xy\in E(F_2)$, we know that in either 1-reduction, the edge added cannot be an $xy$-edge. So $mn\neq xy$ (respectively $mr\neq xy$). By Lemma \ref{lem:2SumMConn}\ref{part:2SumMConn_Sum}, we can 2-sum $F_1$ with either $(F_2)_s^{mn}$ or $(F_2)_s^{mr}$, as relevant, to obtain the $\M$-connected graph $(G')^*=(G_v^{rs})_s^{mn}$ (respectively $(G')^*=(G_v^{rs})_s^{mr}$). We can then perform a length 1-extension to $(G')^*$ to obtain a new graph $G^*=G_v^{st}$ (respectively  $G^*=G_s^{vm}$), as illustrated in Figure \ref{fig:FeasibilityProof_sAdmissible}.
%		
%***********************************
% Feasibility proof when F_2 <> K_4
%***********************************
\begin{figure}[t]
\centering

%***********************************
% Feasibility proof special case G_v^{st}
%***********************************
\subfigure[Construction of $G^*=G_v^{st}$ from $G'$ when $s$ is admissible in $G'$ onto the edge $mn$.]
	{
	\includegraphics{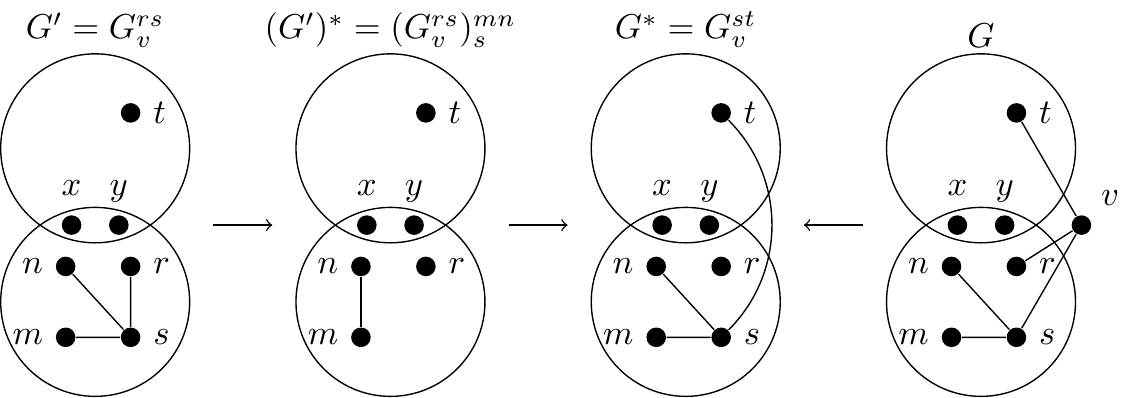}
	}

%***********************************
% Feasibility proof special case G_s^{vm}
%***********************************
\subfigure[Construction of $G^*=G_s^{vm}$ from $G'$ when $s$ is admissible in $G'$ onto the edge $mr$.]
	{
	\includegraphics{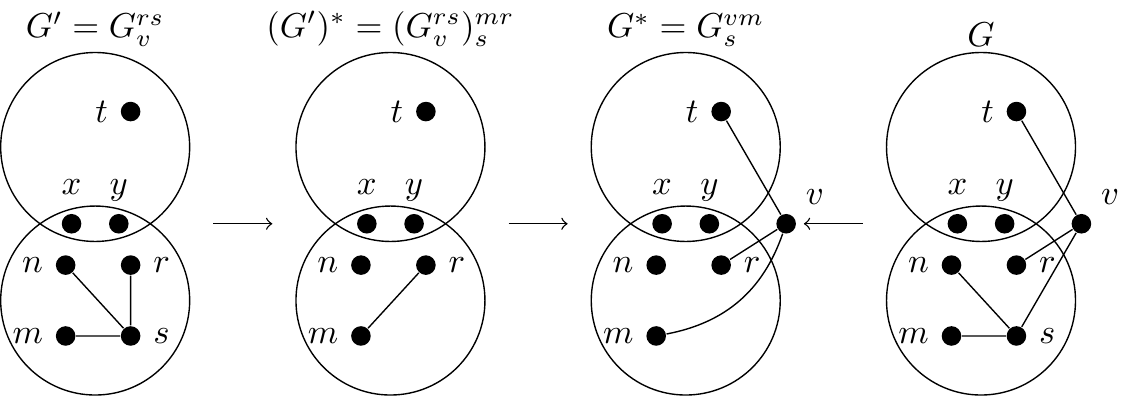}
	}
\caption{Constructions of $G^*$ from $G'=G_v^{rs}$.}
\label{fig:FeasibilityProof_sAdmissible}
\end{figure}
Lemma \ref{lem:MixedMConnectedPreservedBy1extensions} implies that $G^*$ is $\M$-connected in both cases. Further, since the node $s$ (respectively $v$) added back to obtain $G^*$ has neighbourhood $N_{G^*}(s)=\{m,n,t\}$ (resp.\ $N_{G^*}(v)=\{m,r,t\}$) where $t\in V(H_1)-V(H_2)$, and either $\{m,n\}\cap (V(H_2)-V(H_1))\neq \emptyset$ in the first case, or $r\in V(H_2)-V(H_1)$ in the second, we know $\{x,y\}$ is not a 2-vertex-cut of $G^*$.

			By our original assumption, $G^*$ is not direction-balanced. So let $(H_1^*,H_2^*)$ be a 2-separation of $G^*$ on some 2-vertex-cut $\{x^*,y^*\}$, with $H_2^*$ length-pure and minimal with respect to inclusion. Since $G$ is direction-balanced, the set $\{x^*,y^*\}$, must separate $\{s,t\}$ from $r$ in $G^*$ when $G^*=G_v^{st}$, and separate $\{v,m\}$ from $n$ in $G^*$ when $G^*=G_s^{vm}$. %TODO - ambiguous. N(v) may intersect {x^*,y^*} (just know {r,m}\neq {x^*,Y^*} by claim, since 1-red onto rm is admissible). so N(v) not really "separated" by {x^*,y^*}
		Since $m,n,r,s\in V(H_2)$, this implies $V(H_2)\cap V(H_1^*)\neq \emptyset$ in both cases. But $\{x,y\}$ and $\{x^*,y^*\}$ are distinct 2-separators in $(G')^*$ and so cannot cross, by Lemma \ref{lem:UnbalancedMixedMConnHasNoCrossing2Sep}. Hence $\{x^*,y^*\}\subset V(H_2)$, which implies $V(H_2^*)\subset V(H_2)$ in both cases. The definition of an end now implies that there is some vertex in $V(H_2)-V(H_2^*)$ which is contained in some end of $G'$ but in no end of $G^*$. Hence $n(G^*)<n(G')$, which contradicts our choice of $G'$.\qedhere
	\end{case}	
	\end{proof}

Claim \ref{clm:FeasibilityProof_IfNotK4ThenFeasible} tells us that the only case left to consider is when $F_2=K_4$. We show that this cannot occur.

\begin{claim}\label{clm:FeasibilityProof_NotK4}
$F_2\neq K_4$.
%If $F_2=K_4$, then $G$ either contains a feasible edge or a feasible node.
\end{claim}

	\begin{proof}
	%TODO
	%(1) F_2 = K_4
	%(2) F_2 \neq K_4
	%(1) and (2) clearly cannot both be true. Hence giving us the contradiction sought, and proving that G must contain either a feasible edge or a feasible node.
	Assume $F_2=K_4$. We shall show that we can find an admissible move in $G$ which contradicts our choice of $G'$. As in Claim  \ref{clm:FeasibilityProof_IfNotK4ThenFeasible}, there are two cases to consider: when $G'=G-e$ and when $G'=G_v^{rs}$.
	
			%CASE 1:G'=G-e 
		\begin{case}	
	$G'=G-e$. Since $G$ is direction-balanced, $e$ must have at least one endvertex, say $r$, in $V(H_2)-V(H_1)$. Denote the remaining vertex in $V(H_2)-\{x,y\}$ by $z$.
	
	There are two possibilities for $e$: either $e=rz$, in which case $e$ is a direction edge in $G$. Or $e\neq rz$, in which case $e=rs$ for some $s\in V(H_1)$. In both cases we can construct a new graph, $G^*$, from $F_1$, by performing a sequence of two 1-extensions. This construction will give $G^*=G-f$ for some $f\in E(G)$. In the first case, when $e=rz$ is a direction edge, this construction gives $G^*=G-f$	where $f=rz$ is a length-edge. Whereas in the second case, when $e=rs$ for some $s\in V(H_1)$, we obtain $G^*=G-f$ where $f=rx$ is a length-edge. These two constructions are illustrated in Figure \ref{fig:FeasibilityProof_K4_e}. 

%***********************************
% G-f
%***********************************
\begin{figure}[t]
\centering

%***********************************
% G-f (when e=rz)
%***********************************
\subfigure[Construction of $G^*=G-rz$ when $e=rz$ is a direction edge in $G$.]
	{
	\includegraphics{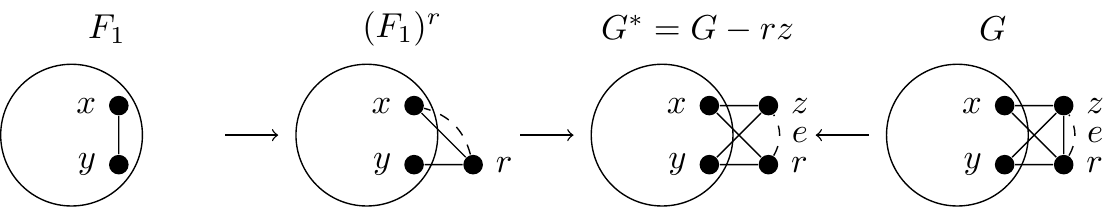}
	\label{fig:FeasibilityProof_K4_e=rz}
	}

%***********************************
% G-f (when e \neq rz)
%***********************************
\subfigure[Construction of $G^*=G-rx$ when $e=rs$ for some $s\in V(H_1)$. Where the type of an edge is not known, it is depicted by a dotted line. Note that this construction also works in the special case where $s\in\{ x,y\}$.]
	{
	\includegraphics{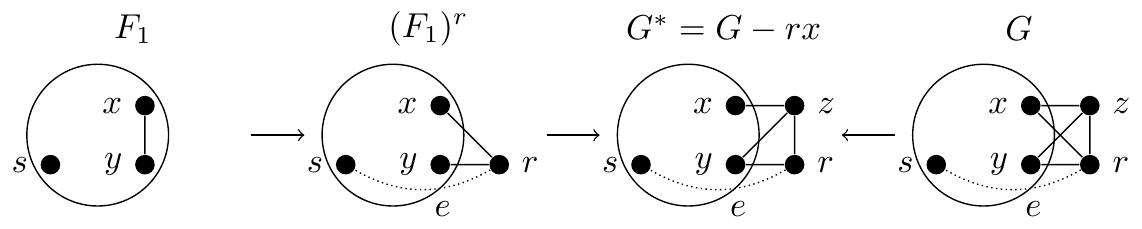}
	\label{fig:FeasibilityProof_K4_e!=rz}
	}
\caption{Constructions of $G^*=G-f$ from $F_1$.}
\label{fig:FeasibilityProof_K4_e}
\end{figure}

		We know $F_1$ is mixed and $\M$-connected, and by Lemma \ref{lem:MixedMConnectedPreservedBy1extensions}, 1-extensions preserve these properties. Hence $G^*$ is mixed and $\M$-connected. By the constructions, it is clear that $\{r,z\}$ is not a length-pure end in $G^*$, whereas it was in $G'$. So any length-pure end of $G^*$ must be contained in $F_1$, and thus is also a length-pure end of $G'$. Hence $n(G^*)< n(G')$, which contradicts our choice of $G'$.
		\end{case}

			\begin{case}
		$G'=G_v^{rs}$. Since $G$ is direction-balanced, $v$ has at least one neighbour in $V(H_2)-V(H_1)$. Denote the remaining vertex in $V(H_2)-\{x,y\}$ by $z$. We shall show that $z$ is admissible in $G$. As before, we can construct an $\M$-connected mixed graph $G^*$ from $F_1$ by a sequence of two 1-extensions. But this time, the graph $G^*$ obtained is a 1-reduction of $G$. There are three different constructions, depending on the structure $G'$.
		
	First, suppose $rs\in E(H_1)$. Then we must have $t\in V(H_2)-V(H_1)$. From Claim \ref{clm:FeasibilityProof_xyNotLengthEdge}, we know that $\{r,s\}\neq\{x,y\}$ so without loss of generality, $x\not\in N_G(v)$. We can then obtain $G^*=G_z^{xy}$ from $F_1$ by the construction shown in Figure \ref{fig:FeasibilityProof_K4_v_rsInS1}
		
%***********************************
% G_z 
%***********************************	
\begin{figure}[t]

%***********************************
% G_z (when r,s\in V_1)
%***********************************
\centering
\subfigure[Construction of $G^*=G_z^{xy}$ when $G'=G_v^{rs}$ and $r,s\in V(H_1)$. Note that this construction also works when $y\in\{r,s\}$.  Where the type of an edge is not know, it is depicted by a dotted line.]
	{
	\includegraphics{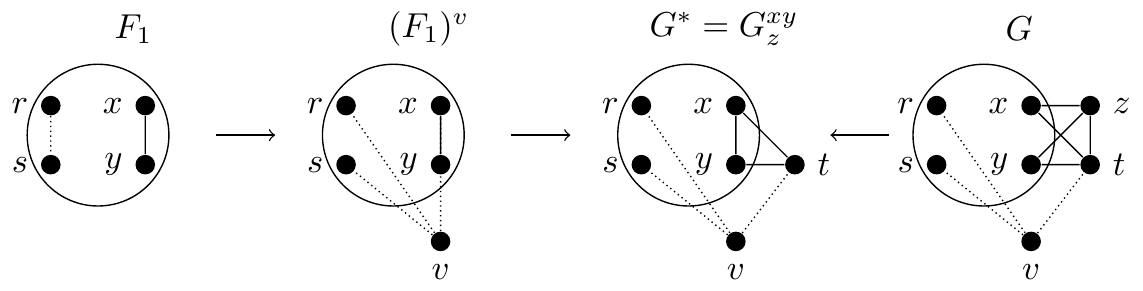}
	\label{fig:FeasibilityProof_K4_v_rsInS1}
	}		

%***********************************
% G_z (when r,s\in V_2-V_1)
%***********************************
\subfigure[Construction of $G^*=G_z^{yv}$ when $G'=G_v^{rs}$, $s=z$ and $r,s\in V(H_2)-V(H_1)$.]
	{
	\includegraphics{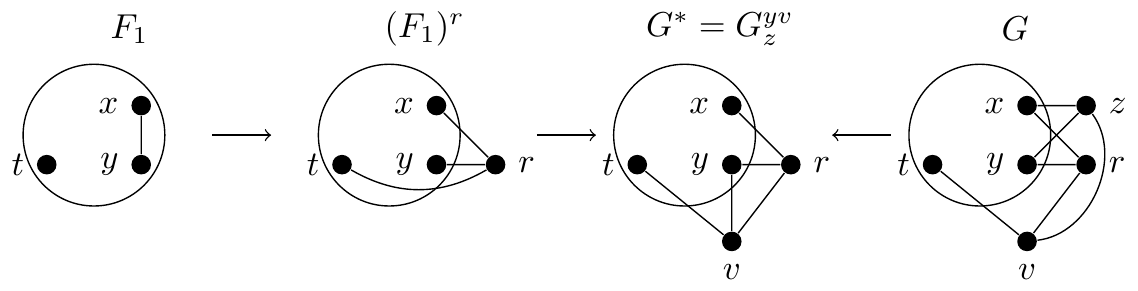}
	\label{fig:FeasibilityProof_K4_v_s=z}
	}

%***********************************
% G_z (when r\in V_2-V_1, s\in V_1\cap V_2)
%***********************************
\subfigure[Construction of $G^*=G_z^{yr}$ when $G'=G_v^{rs}$, $s=y$ and $r\in V(H_2)-V(H_1)$.]
	{
	\includegraphics{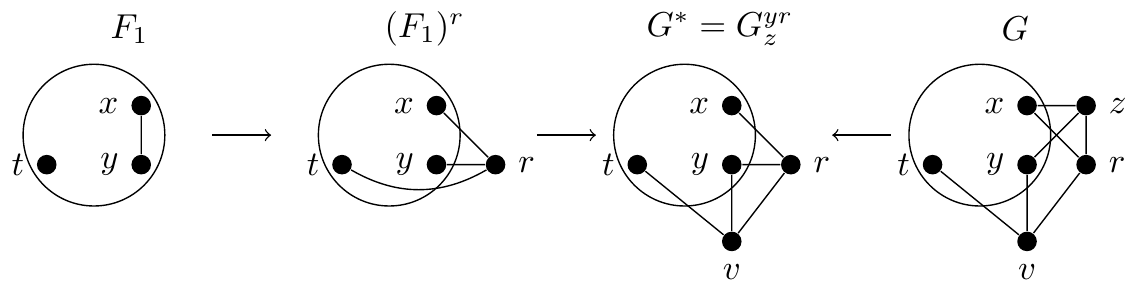}
	\label{fig:FeasibilityProof_K4_v_s=x}
	}
\caption{Constructions of $G^*=G_z$ from $F_1$.}
\label{fig:FeasibilityProof_K4_v}		
\end{figure}
			
			Second, suppose $rs\in E(H_2)$. If $v$ is mixed, then Lemma \ref{lem:UnbalancedLengthRedToDirectionRed}, implies that $G$ has an admissible direction 1-reduction at $v$ onto $rs$, which contradicts our original choice of $G'$. Hence $v$ is length-pure, and so, since $G$ is direction-balanced, $t\in V(H_1)-V(H_2)$. Either both $r,s\in V(H_2)-V(H_1)$, in which case $s=z$, and we obtain $G^*=G_z^{yv}$ by the construction shown in Figure \ref{fig:FeasibilityProof_K4_v_s=z}. Or, relabelling if necessary, $r\in V(H_2)-V(H_1)$ and $s\in V(H_1)\cap V(H_2)$, so without loss of generality $s=y$, and we construct $G^*=G_z^{yr}$	 as shown in Figure \ref{fig:FeasibilityProof_K4_v_s=x}.
		
 In all three cases, $V(G^*)-V(F_1)$ is not a length-pure end of $G^*$, whereas $V(G')-V(F_1)$ was a length-pure end of $G'$. So any length-pure end of $G^*$ is contained in $V(F_1)$, and must also be a length-pure end of $G'$. Thus $n(G^*)< n(G')$, which contradicts our choice of $G'$.\qedhere
		\end{case}
	\end{proof}
	
Clearly Claims \ref{clm:FeasibilityProof_IfNotK4ThenFeasible}	and \ref{clm:FeasibilityProof_NotK4} cannot both hold. Hence our original assumption is wrong, and $G$ contains either a feasible edge or a feasible node.
\end{proof}
 
Theorem \ref{thm:DbalMConnHasFeasibleMove}, together with the fact that edge additions, 1-extensions, and 2-sums with direction-pure $K_4$'s preserve $\M$-connectivity (Lemmas \ref{lem:MixedMConnectedPreservedByEdgeAdditions}, \ref{lem:MixedMConnectedPreservedBy1extensions} and \ref{lem:2SumMConn}\ref{part:2SumMConn_Sum} respectively), and that these operations also preserve being direction-balanced, gives us the following inductive construction:

\begin{thm}\label{thm:RecursiveConstructionDBalMConn}
Let $G$ be a mixed graph. Then $G$ is a direction-balanced, $\M$-connected mixed graph if and only if $G$ can be obtained from $K_3^+$ or $K_3^-$ by a sequence of edge additions, 1-extensions and 2-sums with direction-pure $K_4$'s.
\end{thm}

%#######################################################################################################
%#######################################################################################################

\section{Global Rigidity of $\M$-connected Graphs} \label{sec:GlobalRigidity}

In order for the graphs constructed in Theorem \ref{thm:RecursiveConstructionDBalMConn} to be globally rigid, we need to know that the operations used preserve global rigidity. Lemmas \ref{lem:OperationsPreservingGlobalRigidity} and \ref{lem:2SumDPureK4PreservesGlobalRigidity}  imply that edge additions and 2-sums with direction-pure $K_4$'s preserve global rigidity, but 1-extensions are more troublesome. By Lemma \ref{lem:OperationsPreservingGlobalRigidity}, a 1-extension on a graph $G$ which deletes an edge $e$ will preserve global rigidity so long as $G-e$ is rigid. Fortunately, $\M$-connected mixed graphs are redundantly rigid, by Lemma \ref{lem:MixedMConnectedIsRigid}, so this condition is always satisfied in our construction. Hence all graphs described in Theorem  \ref{thm:RecursiveConstructionDBalMConn} are globally rigid, which gives us the following result:

\begin{lem}\label{lem:Conclusion_MConnDBalImpliesGloballyRigid}
Let $(G,p)$ be a generic mixed framework, and suppose $G$ is $\M$-connected and direction-balanced. Then $G$ is globally rigid.
\end{lem}

By Lemma \ref{lem_JJK:NecessaryConditionsForGlobalRigidity}, all generic, globally rigid direction-length frameworks are mixed and direction-balanced. So Lemma \ref{lem:Conclusion_MConnDBalImpliesGloballyRigid} implies that global rigidity is a generic property of $\M$-connected direction-length frameworks. Hence characterising global rigidity for this class:

\begin{thm}\label{thm:MConnectedGloballyRigidIffDBal}
Let $p$ be a generic realisation of an $\M$-connected mixed graph $G$. Then $(G,p)$ is globally rigid if and only if $G$ is direction-balanced. 
%OLD: An $\M$-connected direction-length graph $G=(V;D,L)$ is globally rigid if and only if it is mixed and direction-balanced.
\end{thm}

This implies that the inductive construction in Theorem \ref{thm:RecursiveConstructionDBalMConn}, is also a construction of all globally rigid $\M$-connected graphs:

\begin{thm}\label{thm:RecursiveConstructionGloballyRigidMConn}
Let $G$ be a graph. Then $G$ is a globally rigid $\M$-connected mixed graph if and only if $G$ can be obtained from $K_3^+$ or $K_3^-$ by a sequence of edge additions, 1-extensions and 2-sums with direction-pure $K_4$'s.
\end{thm}

\section{Concluding Remarks}\label{sec:Further}

There exist globally rigid mixed graphs whose rigidity matroid is not connected, so these results do not characterise global rigidity in general. In particular, $\M$-connected mixed graphs satisfy three properties which are not necessary for global rigidity: they contain at least two length edges, have minimum degree three, and every direction edge is redundant.

Servatius and Whiteley \cite{SW_PlaneCAD} showed that a mixed graph with a single length edge is globally rigid if and only if it is rigid. So to obtain a full characterisation of global rigidity, it remains to characterise it for graphs with $|L|\geq 2$. Our result succeeds in doing this for a large class of such graphs. %%The accompanying paper, \cite{CJK_GlobalRigidityOfDLFrameworks}, builds upon this result to obtain a full characterisation of global rigidity for direction-length graphs; as well as describing the structure of the only remaining class of graphs for which it is not known whether global rigidity is a generic property.

%#######################################################################################################
%#######################################################################################################

\section*{Acknowledgements}

I would like to thank Bill Jackson for suggesting this problem and for helpful discussions regarding it.

%#######################################################################################################
%#######################################################################################################

\bibliographystyle{plain}
\bibliography{References}

\begin{thebibliography}{10}

\bibitem{AR_RigidityOfGraphs}
L.\ Asimow and B.\ Roth.
\newblock Rigidity of graphs {II}.
\newblock {\em J. Math. Anal. Appl.}, 68:171--190, 1979.

\bibitem{BJ_LengthCircuits}
A.R.\ Berg and T.\ Jord{\'a}n.
\newblock A proof of {C}onnelly's conjecture on 3-connected circuits of the
  rigidity matroid.
\newblock {\em Journal of Combinatorial Theory, Series B}, 88:78--89, 2003.

\bibitem{CH_PortOraclesForMatroids}
C.R.\ Coullard and L.\ Hellerstein.
\newblock Independence and port oracles for matroids, with an application to
  computational learning theory.
\newblock {\em Combinatorica}, 16(2):189--208, 1996.

\bibitem{G_L_CharacterisingGlobalRigidity}
S.J.\ Gortler, A.D.\ Healy, and D.P.\ Thurston.
\newblock Characterizing generic global rigidity.
\newblock {\em American Journal of Mathematics}, 132(4):897--939, 2010.

\bibitem{JJ_LengthConnected}
B.\ Jackson and T.\ Jord{\'a}n.
\newblock Connected rigidity matroids and unique realizations of graphs.
\newblock {\em Journal of Combinatorial Theory, Series B}, 94:1--29, 2005.

\bibitem{JJ_MixedCircuits}
B.\ Jackson and T.\ Jord{\'a}n.
\newblock Globally rigid circuits of the direction-length rigidity matroid.
\newblock {\em Journal of Combinatorial Theory, Series B}, 100:1--22, 2010.

\bibitem{JJ_DL_GlobalOperations}
B.\ Jackson and T.\ Jord\'{a}n.
\newblock Operations preserving global rigidity of generic direction-length
  frameworks.
\newblock {\em International J. Comput. Geom. Appl.}, 20:685--708, 2010.

\bibitem{JK_DL_BoundedFrameworks}
B.\ Jackson and P.\ Keevash.
\newblock Bounded direction-length frameworks.
\newblock {\em Discrete Comput. Geom.}, 46(1):48--71, 2011.

\bibitem{JK_NecessaryConditions}
B.\ Jackson and P.\ Keevash.
\newblock Necessary conditions for the global rigidity of direction-length
  frameworks.
\newblock {\em Discrete Comput. Geom.}, 46(1):72--85, 2011.

\bibitem{Ox_Matroids}
J.G.\ Oxley.
\newblock {\em Matroid Theory}.
\newblock Oxford University Press, 2nd edition, 2011.

\bibitem{SW_PlaneCAD}
B.\ Servatius and W.\ Whiteley.
\newblock Constraining plane configurations in computer-aided design:
  combinatorics of directions and lengths.
\newblock {\em SIAM J. Discrete Math.}, 12(1):136--153 (electronic), 1999.

\bibitem{W_MatroidsFromDiscAppGeom}
W.\ Whiteley.
\newblock Some matroids from discrete applied geometry.
\newblock In {\em Matroid Theory}, volume 197 of {\em Contemporary Math.},
  pages 171--311. American Mathematical Society, 1996.

\end{thebibliography}

\end{document}